\documentclass[12pt]{amsart}
\usepackage{amssymb,amsthm,enumerate, wasysym}
\usepackage{array,arydshln}
\usepackage{lscape}
\usepackage{pstricks, pst-plot, pst-node}
\usepackage{hyperref}                         
\usepackage[all]{xy}

\usepackage{tikz}                                 
\usetikzlibrary{angles}
\usetikzlibrary{decorations.markings,calc}
\usetikzlibrary{cd}

\usepackage{enumitem}   

\usepackage[textwidth=2cm]{todonotes}
\reversemarginpar

\input xy
\xyoption{all}

\hypersetup{pdftex,                           
bookmarks=true,
pdffitwindow=true,
colorlinks=true,
citecolor=black,
filecolor=black,
linkcolor=black,
urlcolor=blue,
hypertexnames=true}

\newcommand{\bG}{{\mathbf{G}}}      

\newcommand{\cO}{{\mathcal{O}}}

\newcommand{\bB}{{\mathbf{B}}}
\newcommand{\bC}{{\mathbf{C}}}

\newcommand{\bH}{{\mathbf{H}}}
\newcommand{\bT}{{\mathbf{T}}}

\newcommand{\bV}{{\mathbf{V}}}

\newcommand{\bb}{{\mathbf{b}}}
\newcommand{\bc}{{\mathbf{c}}}
\newcommand{\bd}{{\mathbf{d}}}

\def\GL{ \text{\rm GL} }
\def\GU{ \text{\rm GU} }

\def\SL{ \text{\rm SL} }
\def\PSL{ \text{\rm PSL} }
\def\PGL{ \text{\rm PGL} }
\def\SU{ \text{\rm SU} }
\def\PSU{ \text{\rm PSU} }

\DeclareMathOperator{\Res}{Res}               

\DeclareMathOperator{\Irr}{Irr}

\newtheorem{thm}{Theorem}[section]
\newtheorem{lem}[thm]{Lemma}

\newtheorem{cor}[thm]{Corollary}
\newtheorem{prop}[thm]{Proposition}

\theoremstyle{theorem}

\theoremstyle{definition}
\newtheorem{exmp}[thm]{Example}

\theoremstyle{remark}
\newtheorem{rem}[thm]{Remark}

\newtheorem{notn}[thm]{Notation}

\begin{document}


\title[On the source algebra equivalence class of blocks with cyclic defect groups, III]
{On the source algebra equivalence class of blocks with cyclic defect groups, III}

\date{\today}

\author{
Gerhard Hiss and  Caroline Lassueur
}
\address{{\sc Gerhard Hiss},  Lehrstuhl f\"ur Algebra und Zahlentheorie, 
RWTH Aachen, 52056 Aachen, Germany.}
\email{gerhard.hiss@math.rwth-aachen.de}
\address{{\sc Caroline Lassueur}, Caroline Lassueur, 
RPTU Kaiserslautern-Land\-au, Fachbereich Mathematik,
67653 Kaiserslautern, Germany and Leibniz Universit\"at Hannover,
Institut f\"ur Algebra, Zahlentheorie und Diskrete Mathematik,
Welfengarten 1, 30167 Hannover, Germany.}
\email{lassueur@mathematik.uni-kl.de}

\keywords{Blocks with cyclic defect groups, source algebra equivalences, 
endo-permutation modules, special linear groups, special unitary groups}

\subjclass[2010]{Primary 20C20, 20C15, 20C33.}
\begin{abstract} 
This series of papers is a contribution to the program of classifying
$p$-blocks of finite groups up to source algebra equivalence, starting
with the case of cyclic blocks.
To any $p$-block $\bB$ of a finite group with cyclic defect group $D$,
Linckelmann associated an invariant $W( \bB )$, which is an indecomposable
endo-permut\-ati\-on module over~$D$, and which, together with the Brauer tree
of~$\bB$, essentially determines its source algebra equivalence class.

In Part II of our series, assuming that $p$ is an odd prime, we reduced the 
classification of the invariants $W( \bB )$ arising from cyclic $p$-blocks $\bB$ of 
quasisimple classical groups to the classification for cyclic $p$-blocks 
of quasisimple quotients of special linear or unitary groups. This objective is 
achieved in the present Part III.
\end{abstract}


\maketitle


\pagestyle{myheadings}
\markboth{On the source algebra equivalence class of blocks with cyclic defect groups, III}
{On the source algebra equivalence class of blocks with cyclic defect groups, III}

\vspace{6mm}

\section{Introduction}\label{sec:Intro}
Let~$p$ be an odd prime.  
The purpose of this article is to determine the 
invariants $W( \bar{\bB} )$ for all cyclic $p$-blocks~$\bar{\bB}$  of quasisimple 
groups with simple quotients~$\PSL_n^{\varepsilon}(q)$. As usual, 
$\varepsilon \in \{ -1, 1 \}$ and $\PSL_n^{\varepsilon}(q) = \PSL_n( q )$ 
if $\varepsilon = 1$, and $\PSL_n^{\varepsilon}(q) = \PSU_n( q )$
if $\varepsilon = -1$. The analogous convention is used for the groups
$\SL_n^{\varepsilon}(q)$ and $\GL_n^{\varepsilon}(q)$.
We collect the notation used in this manuscript, and in Parts~{I}
and~{II} of this series of articles, in form of a glossary in 
Section~\ref{sec:Glossary}.
%

Here is the first main result of our analysis.

\begin{thm}
\label{MainTheoremSLNSUNI}
Let $n \geq 2$ be an integer, $q$ a prime power and
$G := \SL_n^\varepsilon( q )$. Let~$p$ be an odd prime with
$p \mid q - \varepsilon$.

Let~$\bar{G}$ denote a central quotient of~$G$ and let $\bar{\bB}$ be a
$p$-block of~$\bar{G}$ with a non-trivial cyclic defect group~$\bar{D}$ of 
order~$p^{l}$.

{\rm (a)} If $G = \SL_n( q )$, then $W( \bar{\bB} ) \cong k$. 

{\rm (b)} If $G = \SU_n(q)$ and 
$p \equiv -1\,\,(\mbox{\rm mod}\,\,4)$, then
$W( \bar{\bB} ) \cong W_{\bar{D}}( A )$ for a subset 
$A \subseteq \{ 1, \ldots , l - 1 \}$ satisfying one of the following conditions.
\begin{itemize}
\item[(i)] $A$ is an interval (possibly the empty set);  
\item[(ii)] $A = [a, l-1] \setminus \{ l - a \}$ for some $1 \leq a \leq l/2$;
\item[(iii)] $A = \{ l - a \} \cup [a, l - 1]$ for an integer $l/2 < a \leq l - 1$.
\end{itemize}

{\rm (c)} If $G = \SU_n(q)$ and $p \equiv 1\,\,(\mbox{\rm mod}\,\,4)$, 
$W( \bar{\bB} ) \cong W_{\bar{D}}( A )$ for a subset
$A \subseteq \{ 1, \ldots , l - 1 \}$ with $|A| \leq 1$. \hfill{\textsquare}
\end{thm}

A proof of Theorem~\ref{MainTheoremSLNSUNI} is provided by
Propositions~\ref{Case1Prop},~\ref{MainCorCase2} and \ref{MainCorCase3}, where 
more precise information is given.

In the notation of the theorem, the cases $p \nmid q - \varepsilon$ have already 
been treated in \cite[Proposition~$6.3$]{HL24} if $p \mid q$, and in 
\cite[Remark~$4.2.3$]{HL25} if $p \nmid q$.

In order to prove this theorem, we assume that $\bar{G} = G/Y$ with a 
non-trivial subgroup $Y \leq Z( G )$, and we let~$\bar{\bB}$ be a cyclic 
$p$-block of~$\bar{G}$. To determine~$W( \bar{\bB} )$, we may assume
that~$Y$ is a $p$-group. Indeed, $\hat{G} := G/O_p(Y)$ is a central extension 
of~$\bar{G}$ by a $p'$-group, and if~$\hat{\bB}$ denotes the $p$-block 
of~$\hat{G}$ dominating~$\bar{\bB}$, then $W( \bar{\bB} ) = W( \hat{\bB} )$ 
by \cite[Lemma~$4.1$]{HL24}.

Let~$\bB$ denote the block of~$G$ dominating~$\bar{\bB}$. 
Then a defect group~$D$ of~$\bB$ is abelian with at most~$2$ generators, but
not necessarily cyclic. Let~$\bc$ denote a Brauer correspondent of~$\bB$ 
in~$C_G(D)$. Now embed~$G$ into $\tilde{G} = \GL^{\varepsilon}_n( q )$, and 
let~$\tilde{\bc}$ denote a block of~$C_{\tilde{G}}( D )$ covering~$\bc$.
Then a defect group~$\tilde{D}$ of~$\tilde{\bc}$ 
is abelian with at most~$3$ generators. This situation is analyzed in detail in 
Section~\ref{TheSpecialLinearAndUnitaryGroups}, which leads to a complete
enumeration of the possibilities for~$W( \bar{\bB} )$ arising from the
blocks~$\bar{\bB}$.

In Section~\ref{Existence} we show that all such possibilities listed in
Theorem~\ref{MainTheoremSLNSUNI} arise for suitable choices of~$q$ and~$n$, 
leading to the following result.

\begin{thm}
\label{MainTheoremSLNSUNII} 
Let $p$ be an odd prime, $l \geq 1$ an integer and let~$\bar{D}$ be a cyclic 
$p$-group of order~$p^{l}$. 

Let $A$ be a subset of $\{ 1, \ldots , l - 1 \}$.  
If $p \equiv -1\,\,(\mbox{\rm mod}\,\,4)$, assume that $A$ satisfies one of the 
conditions {\rm (i)--(iii)} of {\rm Theorem~\ref{MainTheoremSLNSUNI}(b)}. If 
$p \equiv 1\,\,(\mbox{\rm mod}\,\,4)$, assume that $|A| = 1$.

Then there is an integer $n \geq 2$, a prime power~$q$ such that
$p \mid q + 1$, a central quotient $\bar{G}$ of $\SU_n( q )$ and a
$p$-block~$\bar{\bB}$ of $\bar{G}$ with cyclic defect group isomorphic
to~$\bar{D}$, such that $W( \bar{\bB} ) \cong W_{\bar{D}}( A )$. \hfill{\textsquare}
\end{thm} 

A proof of Theorem~\ref{MainTheoremSLNSUNII} is provided in
Propositions~\ref{Existence1Prop}--\ref{Existence3Prop}, which in fact give 
more precise information.

\begin{rem}
{\rm 
Let $p$ be an odd prime, $l \geq 1$ an integer and let~$\bar{D}$ be a cyclic 
$p$-group of order~$p^{l}$. 

If $l \geq 2$, the number of intervals in $[1, l - 1]$, including the empty one, 
equals $l(l - 1)/2 + 1$. The number of subsets of $[1, l - 1]$ as in 
Theorem~\ref{MainTheoremSLNSUNI}(b)(ii) or~(iii) equals $l - 1$, and of these, 
two are intervals, if $l \geq 3$.

Thus the number of isomorphism classes of endo-permutation $k\bar{D}$-modules of 
the form $W( \bar{\bB} )$ arising in Theorem~\ref{MainTheoremSLNSUNI} 
equals~$l$, if $p \equiv 1\,\,(\mbox{\rm mod}\,\,4)$, and $l(l+1)/2 - 2$ if 
$l \geq 3$ and $p \equiv -1\,\,(\mbox{\rm mod}\,\,4)$.

On the other hand, the number of isomorphism classes of endo-permutation 
$kD$-modules equals $2^l$ since~$p$ is odd. 
}\hfill{\textsquare}
\end{rem}

\section{Preliminaries}
\label{MinimalGeneratingSets}
In this section,~$G$ denotes a finite group, and $p$ a prime.

By~$r(G)$ we denote the smallest size of a generating set of~$G$. The following 
standard results on $r(G)$ are stated without proof.

\begin{lem}
\label{MinimalGeneratingSet}
{\rm (a)} If $1 \rightarrow L \rightarrow G \rightarrow H \rightarrow 1$ is a
short exact sequence of finite groups, then $r(G) \leq r(L) + r(H)$.

{\rm (b)} If~$L$ and~$H$ are finite $p$-groups, then $r(L \times H) = r(L) + r(H)$.

{\rm (c)} If~$G$ is abelian and $H \leq G$, then $r(H) \leq r(G)$.
\end{lem}

We will also need the following supplement to \cite[Lemma~$2.2$]{HL24}. Although
this is well known, we include a proof for the convenience of the reader.

\begin{lem}
\label{CentralizersModZ}
Let $Z \leq Z(G)$ be a $p$-group and let $t \in G$ be a $p$-element.
Then the index of $C_G( t )/Z$ in $C_{G/Z}( tZ )$ is a $p$-power.
In particular, if $N_G(C_G(t))/C_G(t)$ is a $p'$-group. Then 
$$C_{G/Z}( tZ ) = C_G( t )/Z.$$
\end{lem}
\begin{proof}
Let $C \leq G$ with $C/Z = C_{G/Z}( tZ )$. Then 
$C_G( t ) \leq C \leq N_G(C_G(t))$, so that the second assertion follows from 
the first.
The map 
$$C \rightarrow Z, c \mapsto ctc^{-1}t^{-1}$$ is a group homomorphism with 
kernel $C_G( t )$. This proves
our first assertion. 
\end{proof}

Finally, we state a result on Brauer pairs needed later on.
\begin{lem}
\label{BrauerPairsLem}
Let $D \leq G$ be an abelian $p$-subgroup. Let $E \leq D$ and $H := C_G( E )$.
Then $D \leq H$. Suppose that~$\bc$ is a $p$-block of~$H$ with defect group~$D$.
Let~$\bB$ denote the $p$-block of~$G$ such that $(E,\bc)$ is a $\bB$-Brauer 
pair. Then~$\bB$ has defect group~$D$.
\end{lem}
\begin{proof}
As $E \leq D$, we have $C_G( D ) \leq H$ and thus $C_G( D ) = C_H( D )$. 
Let~$\bc_0$ denote a Brauer correspondent of~$\bc$ in $C_H( D )$. Then~$\bc_0$
has defect group~$D$ by \cite[Corollary~$6.3.12$]{LinckBook}. Hence $(D,\bc_0)$
is a maximal $\bB$-Brauer pair in~$G$ by \cite[Theorem~$3.10$]{AlpBro}. In 
turn,~$D$ is a defect group of~$\bB$ by \cite[Theorem~$6.3.7$]{LinckBook}.
\end{proof}

We adopt a common and useful diction in a Clifford theory situation. Namely, 
if~$N$ is a normal subgroup of~$G$ and if~$\chi$ and~$\psi$ are irreducible
$K$-characters of~$G$, respectively~$N$, we say that~$\psi$ lies below~$\chi$
or that~$\chi$ lies above~$\psi$, if~$\psi$ is a constituent of the restriction
of~$\chi$ to~$N$.

\section{Analysis}\label{sec:Analysis}
\label{TheSpecialLinearAndUnitaryGroups}

Here, we analyze the relevant configurations, leading to a proof of 
Theorem~\ref{MainTheoremSLNSUNI}.

\subsection{The groups}
\label{TheGroups}
Let us begin by introducing the groups and some corresponding notation,
used throughout this section.

In order not to overload the notation, we slightly change the global conventions
used in~\cite{HL25}. As our focus is on the special linear and unitary groups,
$\SL_n^{\varepsilon}(q)$ is denoted by~$G$, and $\GL_n^{\varepsilon}(q)$
by~$\tilde{G}$; see Notation~\ref{HypoI} below. Also,~$D$ does not necessarily 
denote a cyclic group, and symbols such as~$D_1$ attain a new significance.

\addtocounter{thm}{1}
\begin{notn}
\label{HypoI}
{\rm 
(i) Let $\varepsilon \in \{ -1, 1 \}$ and $n \geq 2$ be integers, $p$ an odd 
prime and~$q$ a power of a prime~$r$ with $p \mid q - \varepsilon$. 
Let~$\mathbb{F}$ denote an algebraic closure of the finite field with~$r$ 
elements. 

(ii) Put $\tilde{\bG} := \GL_n( \mathbb{F} )$ and $\bG :=
\{ g \in \tilde{\bG} \mid \det(g) = 1 \}$. Let $\bV := \mathbb{F}^n$, the 
natural vector space for~$\tilde{\bG}$.

(iii) Let $F := F_{\varepsilon}$ denote a Steinberg morphism of~$\tilde{\bG}$ such 
that $\tilde{G} = \tilde{\bG}^F = \GL_n^\varepsilon( q )$. Then~$\bG$ is 
$F$-stable and $G = \bG^F = \SL_n^\varepsilon( q )$. Put $Z := Z( G )$.
Notice that $|Z| = \gcd( q - \varepsilon, n )$.

(iv) Let $\delta = 1$ if $\varepsilon = 1$, and $\delta = 2$, if 
$\varepsilon = -1$, and put $V := \mathbb{F}_{q^\delta}^n \subseteq \bV$. 
Then~$V$ is the 
natural vector space for~$\tilde{G}$. (This is consistent with the notation 
introduced in \cite[Subsection~$3.1$]{HL25}.)

(v) Let~$p^a$ and~$p^b$ 
denote the highest powers of~$p$ dividing $q - \varepsilon$,
respectively~$n$, and put $c := \min\{a,b\}$. Then~$p^c$ is the highest
power of~$p$ dividing~$|Z|$.
}\hfill{\textsquare}
\end{notn}

\addtocounter{subsection}{1}
\subsection{Preliminaries on~$\tilde{G}$}
Before we continue, we record two results on~$\tilde{G}$. The first of these is
purely group theoretical in nature.

\addtocounter{thm}{1}
\begin{lem}
\label{CentralizersAndDeterminants}
{\rm (a)} Let $\tilde{\bH} \leq \tilde{\bG}$ be a regular subgroup. Then 
$\bH := \bG \cap \tilde{\bH}$ is a regular subgroup of~$\bG$ and
$$[\tilde{H}\colon\!H] = q - \varepsilon.$$

{\rm (b)} Let $s \in \tilde{G}$ be semisimple. Then
$$[C_{\tilde{G}}( s )\colon\!C_G( s )] = q - \varepsilon.$$
In particular, if $s \in G$, the $G$-conjugacy class of~$s$ is a
$\tilde{G}$-conjugacy class.
\end{lem}
\begin{proof}
(a) The inclusion $\bG \rightarrow \tilde{\bG}$ is a regular embedding; see
\cite[Definition~$1.7.1$]{GeMa}. Moreover,  $C_{\tilde{\bG}}( s )$ is a regular 
subgroup of~$\tilde{\bG}$.
The claims follow from \cite[Lemma~$2.5.3$]{HL25}.
\end{proof}

For the following result recall the notion~$r(H)$, introduced in 
Subsection~\ref{MinimalGeneratingSets} for a finite group~$H$.
\begin{lem}
\label{RankTwoDefectGroup}
Let~$\tilde{D}$ be an abelian defect group of some $p$-block of~$\tilde{G}$
with $r(\tilde{D}) = 2$. Then $p^{2a} \mid |\tilde{D}|$.
\end{lem}
\begin{proof}
This follows from~\cite[Theorem~(3C)]{fs2}.
\end{proof}

\addtocounter{subsection}{2}
\subsection{The blocks and their defect groups}
We introduce the principal object of our study and set up further notation. 

\addtocounter{thm}{1}
\begin{notn}
\label{HypoII}
{\rm 
(i) Let $Y \leq Z$ be a $p$-group, put $\bar{G} := G/Y$, and write 
$\bar{\ }: G \rightarrow \bar{G}$ for the natural epimorphism. 

(ii) Let~$\bar{\bB}$ denote a $p$-block of~$\bar{G}$ with a non-trivial cyclic 
defect group. 

(iii) Let~$\bB$ denote the $p$-block of~$G$ dominating~$\bar{\bB}$ and let~$D$ 
be a defect group of~$\bB$. (Then $\bar{D} = D/Y$ is a defect group 
of~$\bar{\bB}$ by \cite[Lemma~$2.4.1$]{HL25}.)

(iv) Choose $t \in D$ with $\bar{D} = \langle \bar{t} \rangle$.

(v) Define the non-negative integer~$c'$ by 
$p^{c'} = |\langle t \rangle \cap Y|$. 
}\hfill{\textsquare}
\end{notn}

Notice that, as~$Y$ is a $p$-group, for any block~$\bB$ of~$G$ there is a 
unique block~$\bar{\bB}$ dominated by~$\bB$; see \cite[Theorem~$5.8.11$]{NaTs}.
Let us record some easy observations.

\begin{lem}
\label{YEqualOpZ}
{\rm (a)} We have $Y \leq Z( \tilde{G} )$.

{\rm (b)} The kernel of the natural epimorphism 
$\langle t \rangle \rightarrow \bar{D}$ equals $\langle t \rangle \cap Y$, so
that $|\bar{D}| = |t|/p^{c'}$.

{\rm (c)} We have $O_p( Z ) = Z \cap D$. Moreover, 
$D = \langle t, Y \rangle = \langle t, O_p(Z) \rangle$. In particular, 
$r(D) \leq 2$.

{\rm (d)} If $\langle t \rangle \cap O_p( Z ) = \{ 1 \}$, then $Y = O_p( Z )$.

{\rm (e)} If $|D| = |t|p^{a}$ then 
$\langle t \rangle \cap O_p( Z ) = \{ 1 \}$ and $p^a \mid n$.

{\rm (f)} If~$D$ is cyclic, then $D = \langle t \rangle$.
\end{lem}
\begin{proof}
(a) and (b) are trivial.

(c) Since $D$ is a defect group, we have $O_p( Z ) \leq D$, implying the first 
assertion. The other assertions follow
from $\bar{D} = \langle \bar{t} \rangle$ and $Y \leq O_p(Z)$.

(d) By~(c) and the assumption, $D = \langle t \rangle \times O_p( Z )$. Since 
$Y \leq O_p( Z )$, the quotient $D/Y$ can only be cyclic if $Y = O_p( Z )$.

(e) We have $|t|p^a = |D| \leq |t||Y| \leq |t||O_p(Z)| \leq |t|p^a$, implying
the claims.

(f) If~$D$ is cyclic and~$\bar{D}$ is non-trivial, the image of a proper subgroup 
of~$D$ is a proper subgroup of~$\bar{D}$. As the image of $\langle t \rangle$
equals~$\bar{D}$, we obtain our claim.
\end{proof}

\addtocounter{subsection}{2}
\subsection{The local configuration}
\label{LocalConfiguration}
Let us have a look at the local situation. We begin by introducing further
notation.

\addtocounter{thm}{1}
\begin{notn}
\label{HypoIII}
{\rm
(i) Put $C := C_{G}( D )$ and $\tilde{C} := C_{\tilde{G}}( D )$. 
Then $C \unlhd \tilde{C}$.

(ii) Let~$\bc$ be a Brauer correspondent of~$\bB$ in~$C$. Then~$D$ is a defect 
group of~$\bc$.

(iii) Let~$\tilde{\bc}$ be a block of~$\tilde{C}$ covering~$\bc$ and
let~$\tilde{D}$ be a defect group of~$\tilde{\bc}$ with $D = C \cap \tilde{D}$.
}\hfill{\textsquare}
\end{notn}

Notice that $C = C_G( t )$ and $\tilde{C} = C_{\tilde{G}}( t )$, as 
$O_p( Z ) \leq Z( \tilde{G} )$ and $D = \langle t, O_p( Z ) \rangle$.
The groups and blocks in question are displayed in the diagram of 
Figure~\ref{DiagramForSLn}. 
The invariant $W( \bar{\bB} )$ will be computed with the help of~$\tilde{\bc}$.

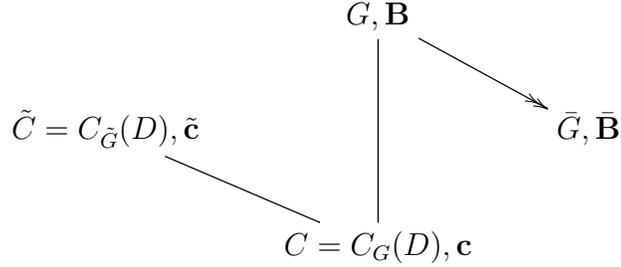
\begin{figure}
\caption{\label{DiagramForSLn} Some subgroups and blocks of $\GL^\varepsilon_n( q )$, I}
$$
\begin{xy}
\xymatrix@C+1pt{
 & G, {\bB} \ar@{-}[dd] \ar@{->>}[rd] & \\
\tilde{C} = C_{\tilde{G}}( D ), \tilde{\bc} \ar@{-}[rd] & & \bar{G}, \bar{\bB} \\
& C = C_G( D ), {\bc} &  
}
\end{xy}
$$
\end{figure}

\begin{lem}
\label{TildeCmoduloC}
The following statements hold.

{\rm (a)} The group $\tilde{C}/C$ is cyclic of order $q - \varepsilon$.

{\rm (b)} The group $\tilde{D}/D$ is cyclic of order dividing~$p^a$. 
	In particular, $\tilde{D}$ is abelian and $r( \tilde{D} ) \leq 
	r(D) + 1 \leq 3$.

{\rm (c)} We have $D = G \cap \tilde{D}$ and $D = G \cap O_p( Z( \tilde{C} ) )$.

{\rm (d)} If $\tilde{D}$ is the Sylow $p$-subgroup of a maximal torus of~$\tilde{G}$,
then $|\tilde{D}| = |D|p^a$.

\end{lem}
\begin{proof}
(a) This follows from Lemma~\ref{CentralizersAndDeterminants}(a). 

(b) We have $\tilde{D}/D = \tilde{D}/(C \cap \tilde{D}) \cong \tilde{D}C/C 
\leq \tilde{C}/C$, so that the first claim is a consequence of~(a). In 
particular,~$\tilde{D}$ is abelian, as $D \leq Z( \tilde{D} )$. The next 
claim follows from the first and Lemma~\ref{MinimalGeneratingSet}(a).

(c) The first assertion is clear, as $C = G \cap \tilde{C}$. 

To prove the second assertion, we first show that 
$D \leq G \cap O_p( Z( \tilde{C} ) )$.
As $D \leq G$, it suffices to show that $D \leq O_p( Z( \tilde{C} ) )$.
Now $D$ is abelian and so $D \leq C_{\tilde{G}}( D ) = \tilde{C}$ and
$[D,\tilde{C}] = \{ 1 \}$. Hence $D \leq Z( \tilde{C} )$, and as~$D$
is a $p$-group, we obtain $D \leq O_p( Z( \tilde{C} ) )$.
To prove the reverse inclusion, first observe that 
$G \cap Z( \tilde{C} ) \leq Z( C )$.
As $G \cap O_p( Z( \tilde{C} ) )$ is a $p$-group, this implies 
$G \cap O_p( Z( \tilde{C} ) ) \leq O_p( Z( C ) )$. Now $O_p( Z( C ) ) = D$ by 
\cite[Lemma~$2.1$]{HL24}, and our proof is complete.

(d) Assume that $\tilde{D}$ is the Sylow $p$-subgroup of the maximal
torus~$\tilde{T}$ of~$\tilde{G}$. Then $D = G \cap \tilde{D} = 
G \cap \tilde{T} \cap \tilde{D}$ is the Sylow $p$-subgroup of the maximal
torus $T := G \cap \tilde{T}$ of $G$. Since $[\tilde{T}\colon\!T] = 
q - \varepsilon$ by Lemma~\ref{CentralizersAndDeterminants}(a), the claim 
follows. 
\end{proof}

\addtocounter{subsection}{2}
\subsection{Analyzing the local configuration}

Recall from Notation~\ref{HypoI} that~$V$ is the natural $n$-dimensional 
$\mathbb{F}_{q^\delta}$-vector space of~$\tilde{G}$. In what follows, we will
make use of the corresponding notation introduced in 
\cite[Subsections~$3.1$,~$3.2$,~$3.4$]{HL25}. 

\addtocounter{thm}{1}
\begin{lem}
\label{MinimalPolynomial}
The minimal polynomial of~$t$ has at most three irreducible factors.
\end{lem}
\begin{proof}
Let~$h$ denote the number of irreducible factors of the minimal polynomial 
of~$t$ acting on~$V$. By the primary decomposition of~$V$ with respect to~$t$, 
we get $\tilde{C} = C_{\tilde{G}}( t ) = 
\tilde{C}_{1} \times \cdots \times \tilde{C}_{h}$, where
each~$\tilde{C}_i$ is a general linear or unitary group, possibly over an
extension field of $\mathbb{F}_{q^{\delta}}$. As $p \mid q - \varepsilon$, we
have $p \mid |Z( \tilde{C}_i )|$ for each $1 \leq i \leq h$.

By Lemma~\ref{MinimalGeneratingSet}(b) we have $r( O_p( Z( \tilde{C} ) ) = h$.
As~$\tilde{D}$ is a defect group of some $p$-block of~$\tilde{C}$, we have 
$O_p( Z( \tilde{C} ) ) \leq \tilde{D}$.
It follows that $h = r( O_p( Z( \tilde{C} ) ) \leq r( \tilde{D} ) 
\leq 3$, the first inequality arising from Lemma~\ref{MinimalGeneratingSet}(c),
the second one from Lemma~\ref{TildeCmoduloC}(b).
\end{proof}

\begin{notn}
\label{HypoIV}
{\rm
(i) Let us write $\Delta_1, \ldots , \Delta_h$ with $h \in \{ 1, 2, 3 \}$ for
the monic, irreducible factors of the minimal polynomial of~$t$.

(ii)
Fix $j \in \{ 1, \ldots , h \}$.
Let $\xi_j \in \mathbb{F}$ denote a root of~$\Delta_j$. Let~$d_j$ denote the 
degree of~$\Delta_j$, and~$n_j'$ its multiplicity in the characteristic 
polynomial of~$t$. Write $n_j' = m_jp^{b_j}$ with non-negative integers $b_j, 
m_j$ and $p \nmid m_j$.

(iii) We choose the notation in such a way that
$d_1 \geq \cdots \geq d_{h}$. If $d_1 = 1$ and $h \geq 2$, we assume
$|\xi_j| \geq |\xi_{j+1}|$ for $1 \leq j < h$.

(iv) Put $V_j := \ker( \Delta_j( t ) )$. Then~$V_j$ is $\tilde{C}$-invariant. Let 
$t_j$ denote the restriction of~$t$ to~$V_j$. (Depending on the context, we 
view~$t_j$ as an automorphism of~$V_j$ or of~$V$.) Then the minimal polynomial 
of~$t_j$ acting on~$V_j$ equals~$\Delta_j$. Moreover, 
$\dim_{\mathbb{F}_{q^\delta}}( V_j ) = n_j$ with $n_j := n_j'd_j = m_jd_jp^{b_j}$.
}\hfill{\textsquare}
\end{notn}

We record some properties of the quantities introduced above.
\begin{lem}
\label{DeterminantLem}
Fix $j$ with $1 \leq j \leq h$. Then the following statements hold.

{\rm (a)} We have $d_j = p^{a_j}$ with a 
non-negative integer~$a_j$; in particular, $n_j = m_j p^{a_j}$ if $b_j = 0$.
Also, $|\xi_j| = p^{a + a_j}$, if $a_j > 0$, and
$|\xi_j| \mid p^a$, if $a_j = 0$. 

{\rm (b)} If $\varepsilon = -1$, we have $\Delta_j = \Delta_j^\dagger$, i.e.\ 
$\xi_j^{-q}$ is a root of~$\Delta_j$. 

{\rm (c)} The highest power of~$p$ dividing 
$(q^{\delta {d_j}} - 1)/(q^{\delta} - 1)$ equals $p^{a_j}$, and we 
define the positive integer $m_j'$ by 
$(q^{\delta {d_j}} - 1)/(q^{\delta} - 1) = m_j'p^{a_j}$; then 
$p \nmid m_j'$. 

{\rm (d)} We have  
$\det(t_j) = \xi_j^{m_jm_j'p^{a_j+b_j}}$.
In particular, $|\det(t_j)| = p^{a-b_j}$ if $a_j > 0$ and $b_j \leq a$.

{\rm (e)} Let~$\kappa_j$ denote the restriction of~$\kappa$ to~$V_j$. Then~$\kappa_j$ 
is a non-degenerate hermitian form if $\varepsilon = -1$. We view 
$I(V_j,\kappa_j)$ as a subgroup of~$\tilde{G}$ in the natural way. Put 
$\tilde{G}_j := {I(V_j,\kappa_j)}$ and 
$\tilde{C}_{j} := C_{\tilde{G}_j}( t_j )$. Then 
$\tilde{G}_j \cong \GL^\varepsilon_{n_j}( q )$ and 
$\tilde{C}_{j} \cong \GL^\varepsilon_{n_j'}( q^{p^{a_j}} )$. 

In particular,
$Z( \tilde{C}_{j} )$ is cyclic and $|Z( \tilde{C}_{j} )|_p = p^{a+a_j}$.
\end{lem}
\begin{proof}
(a) See \cite[Lemma~$4.1.1$(c)]{HL25}.

(b) This is \cite[Lemma~$3.2.2$]{HL25}.

(c) This is standard.

(d) See \cite[Lemma~$4.1.1$(e)]{HL25}.

(e) The fact that~$\kappa_j$ is non-degenerate if $\varepsilon = -1$ follows
from~(b).
\end{proof}
Recall from Notation~\ref{HypoIII}(iii) that $\tilde{D}$ denotes a defect group 
of the block $\tilde{\bc}$ of
$\tilde{C} = \tilde{C}_1 \times \cdots \times \tilde{C}_h$. There is a 
decomposition of~$\bc$ into a tensor products of blocks of $\tilde{C}_j$ with
defect groups $\tilde{D}_j \leq \tilde{C}_j$ for $1 \leq j \leq h$ such that 
$\tilde{D} \cong \tilde{D}_1 \times \cdots \times \tilde{D}_h$.

\begin{lem}
\label{AnalysingTheDeltas}
Assume the terminology introduced in 
{\rm Notations~\ref{HypoI}, \ref{HypoII}, \ref{HypoIII}, \ref{HypoIV}} and 
{\rm Lemma \ref{DeterminantLem}}.

{\rm (a)} We have $|t| \leq p^{a + a_1}$ with equality if $a_1 > 0$. 

{\rm (b)} We have
\begin{equation}
\label{EstimatingTildeD}
|\tilde{D}| \leq |D|p^a \leq |t|p^{2a} \leq p^{3a + a_1}.
\end{equation}
If $|\tilde{D}| = p^{3a + a_1}$, then $|D| = |t|p^a$ and $|t| = p^{a+a_1}$
(even if $a_1 = 0$).

{\rm (c)} Suppose that the minimal polynomial of~$t$ is irreducible and that
$D \not\leq Z$. 
Then~$\tilde{D}$ is cyclic, so that~$D$ is cyclic, and~$\bc$ is covered by a 
cyclic block. Moreover, $a_1 > 0$ and $b_1 = a$. In particular, 
$n = m_1 p^{a+a_1}$ with $p \nmid m_1$.

{\rm (d)} Suppose that the minimal polynomial of~$t$ has exactly 
two irreducible factors. Then $\tilde{D}_1$ and $\tilde{D}_2$ are cyclic and
$b_1 = b_2 = 0$. Moreover, the following statements hold.

\begin{itemize} 
\item[(i)] If $a_1 > a_2$, then $\langle t \rangle \cap O_p(Z) = \{ 1 \}$ and
$Y = O_p(Z)$. In particular, $c' = 0$. Moreover, $a_2 = c$ and 
$|D| = p^{a + a_1 + c}$.

\item[(ii)] If $a_1 = a_2 > 0$, then $|D| = p^{a + 2a_1}$,
$c' = c - a_1 < a$ and $Y = O_p(Z)$. 

\item[(iii)] If $a_1 = a_2 = 0$, then~$D$ is cyclic of order~$p^a$. In particular,
$D = \langle t \rangle$ and $|t| = p^a$.

\item[(iv)] If $a_2 = 0$, then $D$ is cyclic.

\item[(v)] We have $c \geq a_2$.

\item[(vi)] Suppose that $a_1 = a_2 = 0$ and that $c' < c$. 
Then $n_1 \neq n_2$.

\item[(vii)] We have $|t_j| = p^{a+a_j}$ for $j = 1, 2$.
\end{itemize}

{\rm (e)} Suppose that the minimal polynomial of~$t$ has exactly
three irreducible factors. Then $b_1 = a_2 = b_2 = a_3 = b_3 = 0$ and 
$|\tilde{D}| = p^{3a + a_1}$. In particular, $p^a \mid n$, $Y = O_p(Z)$
and $c' = 0$.
Thus $|\bar{D}| = |t| = p^{a+a_1}$.

If $a_1 > 0$, then $n_2 \neq n_3$ and $N_{\tilde{G}}( \tilde{C} ) = 
N_{\tilde{G}_1}( \tilde{C}_{1} ) \times \tilde{G}_2 \times \tilde{G}_3$.

{\rm (f)} In all cases, $|\tilde{D}| = |D|p^a$, $|t| = p^{a+a_1}$ and 
$|\bar{D}| = p^{a+a_1-c'}$.

{\rm (g)} We have $C_{\tilde{G}}( \tilde{D} ) \leq C_{\tilde{G}}( D )$ and
$N_{\tilde{G}}( \tilde{D} ) \leq N_{\tilde{G}}( D )$, with equality in either 
instance, if the minimal polynomial of~$t$ is reducible.
\end{lem}
\begin{proof}
(a) Lemma~\ref{DeterminantLem}(a) implies that $|t_1| \mid p^{a + a_1}$ with 
equality if $a_1 > 0$, and $|t_j| \mid p^{a + a_1}$ for $1 \leq j \leq h$. 

(b) The first inequality 
of~(\ref{EstimatingTildeD}) follows from Lemma~\ref{TildeCmoduloC}(b), the
second one from $D = \langle t, O_p(Z) \rangle$ and $|O_p(Z)| = p^c \leq p^a$, 
and the last one from~(a).

The last assertions are clear from~(\ref{EstimatingTildeD}).

(c) Here, $D = G \cap O_p( Z ( \tilde{C}_1 ) )$. If $a_1 = 0$, then
$\tilde{C}_1 = \tilde{G}$ and $D \leq Z$, a case we have excluded. Thus
$a_1 > 0$ and~$D$ is cyclic of order $p^{a+a_1}$ by~(a) and 
Lemma~\ref{YEqualOpZ}(f). If $\tilde{D}_1$ is not cyclic, then 
$p^{a+a_1} = |D| \geq |\tilde{D}|/p^a \geq p^{a+2a_1}$, where the latter 
estimate follows from Lemma~\ref{RankTwoDefectGroup}, applied to~$\tilde{C}_1$. 
This contradiction shows that~$\tilde{D}$ is cyclic. In particular, 
$|\tilde{D}| = p^{a+a_1+b_1}$, and~$\tilde{D}$ is the Sylow $p$-subgroup of a 
maximal torus of~$\tilde{G}$; see \cite[Corollary~$3.6.2$]{HL25}. 
Lemma~\ref{TildeCmoduloC}(d) implies that 
$p^{a+a_1} = |D| = |\tilde{D}|/p^a = p^{a_1+b_1}$, and thus $b_1 = a$ as
claimed.

(d) Recall that $3 \geq r( \tilde{D} ) = r( \tilde{D}_1 ) + r( \tilde{D}_2 )$.
If $\tilde{D}_j$ is non-cyclic, then $|\tilde{D}_j| \geq p^{2(a+a_j)}$
by Lemma~\ref{RankTwoDefectGroup}, applied to~$\tilde{C}_j$ for $j = 1, 2$.

Suppose first that $\tilde{D}_1$ is non-cyclic. Then~$\tilde{D}_2$ is cyclic, 
and, using~(\ref{EstimatingTildeD}) and \cite[Corollary~$3.6.2$]{HL25}, we get
$$p^{3a + a_1} \geq |\tilde{D}| \geq p^{2(a+a_1) + a + a_2 + b_2}
= p^{3a + 2a_1+  a_2 + b_2},$$
which implies $a_1 = a_2 = b_2 = 0$ and also that $|\tilde{D}| = p^{3a}$.
Now $D = G \cap O_p( Z( \tilde{C} ) )$. Since $a_1 = a_2 = 0$, we have
$|O_p( Z( \tilde{C} ) )| = p^{2a}$. Clearly, $O_p( Z( \tilde{C} ) ) \not\leq G$
since $b_2 = 0$, i.e.\ $p \nmid n_2$. Hence $|D| < p^{2a}$, contradicting
$|D| \geq |\tilde{D}|/p^a$.

Suppose then that $\tilde{D}_2$ is non-cyclic. 
Then
$$p^{3a + a_1} \geq |\tilde{D}| \geq p^{a + a_1 + b_1 + 2(a+a_2)}
= p^{3a + a_1 + b_1 + 2 a_2},$$
which implies $b_1 = a_2 = 0$ and also that $|\tilde{D}| = p^{3a + a_1}$.
If $a_1 = 0$, an analogous argument as above leads to a contradiction. So assume
that $a_1 > 0$ in the following. By~(b) and Lemma~\ref{YEqualOpZ}(e), we obtain
$p^a \mid n$. We claim that $b_2 = 0$, i.e.\ $p \nmid n_2'$. Indeed,
$|\det(t_1)| = p^a$ by Lemma~\ref{DeterminantLem}(d), and 
$\det(t_2) = \xi_2^{m_2m_2'p^{b_2}} = \xi_2^{n_2'm_2'}$, as $a_2 = 0$. From 
$\det( t_1 ) \det(t_2) = \det(t) = 1$, we conclude $p^a = |\det(t_1)| 
= |\det(t_2)|$, and thus $p \nmid n_2'$, which is our claim. Now 
$n = n_1'p^{a_1} + n_2'$ and $p \nmid n_2'$. As $a_1 > 0$, we conclude that 
$p \nmid n$, a contradiction.

We next show that $b_1 = b_2 = 0$, assuming first that $a_1 > 0$. Then 
$|t| = p^{a+a_1}$ by~(a). Now $D = \langle t, O_p(Z) \rangle$ and 
$|O_p(Z)| = p^c$, so that $|D| \leq |t||O_p(Z)| = p^{a+a_1+c}$.
Since $\tilde{D}_1$ and $\tilde{D}_2$ are cyclic, we get $|\tilde{D}| =
p^{2a + a_1 + b_1 + a_2 + b_2}$. As~$\tilde{D}$ is a Sylow $p$-subgroup
of~$\tilde{G}$ by \cite[Corollary~$3.6.2$]{HL25}, 
Lemma~\ref{TildeCmoduloC}(d) implies that
\begin{equation}
\label{FirstDEstimate}
|D| = p^{a + a_1 + b_1 + a_2 + b_2},
\end{equation}
and thus
\begin{equation}
\label{cEstimate}
c \geq b_1 + a_2 + b_2.
\end{equation}
Since $a \geq c \geq b_1 + a_2 + b_2$, we obtain $a \geq b_j$ for $j = 1, 2$.
If also $a_2 > 0$, we have $|\det(t_j)| = q^{a-b_j}$ for $j = 1, 2$ by
Lemma~\ref{DeterminantLem}(d). From $|\det(t_1)| = |\det(t_2)|$, we obtain 
$b_1 = b_2$, if $a_1, a_2 > 0$. 

Assume now that $a_1 > a_2 > 0$. Then 
$n = m_1 p^{a_1 + b_1} + m_2 p^{a_2 + b_2}$. As $a_1 + b_1 > a_2 + b_1
= a_2 + b_2$, we obtain $a_2 + b_2 = b \geq c \geq b_1 + a_2 + b_2$. It follows
that $b_2 = b_1 = 0$.

Next assume that $a_1 > a_2 = 0$. Then $|\det(t_1)| = p^{a-b_1}$ by
Lemma \ref{DeterminantLem}(d). Since $\det(t_2) = \xi_2^{m_2m_2'p^{b_2}}$ and 
$|\xi_2| \mid p^a$, it follows that $b_2 = 0$ if $b_1 = 0$. Now 
$n = m_1 p^{a_1 + b_1} + m_2 p^{b_2}$. If $a_1 + b_1 > b_2$, then 
$b_2 = b \geq c \geq b_1 + b_2$ and hence $b_1 = 0$. As noticed above, this 
implies $b_2 = 0$.  Otherwise, $a_1 + b_1 \leq b_2$. If $a = b_1$, we obtain 
$b_2 = 0$ from $c \geq b_1 + b_2 = a + b_2 \geq a \geq c$. But then 
$p \nmid n$, contradicting $c \geq b_1 + b_2 = a > 0$. Hence $a \neq b_1$, 
i.e.\ $|\det(t_1)| = p^{a-b_1} > 1$. Writing
$|\xi_2| = p^e$, we obtain $a - b_1 = e - b_2$ from $|\det(t_1)| = |\det(t_2)|$. 
It follows that $a - b_1 = e - b_2 \leq e - a_1 - b_1$, and so
$e \geq a + a_1 > a$, a contradiction.

Finally assume that $a_1 = a_2 > 0$. As $D = G \cap O_p( Z ( \tilde{C} ) )$ and 
$O_p( Z ( \tilde{C} ) ) = \langle t_1 \rangle \times \langle t_2 \rangle$ we get
$D \cong \{ (t_1^i, t_2^j) \mid \det( t_1^j ) = \det( t_2^i )^{-1} \}$.
Thus $|D| = p^{a + a_1 + a_2 + b_1}$, as the kernel of the map
$\det \colon \langle t_2 \rangle \rightarrow \mathbb{F}^*$ has order
$p^{a_2 + b_2}$. From $p^{a + a_1 + a_2 + b_2} = |D| \geq |\tilde{D}|/p^a =
p^{a + a_1 + b_1 + a_2 + b_2}$ we obtain $b_1 = 0$, and thus $b_2 = 0$.

Let us now assume that $a_1 = a_ 2 = 0$. Simultaneously to the proof
hat of $b_1 = b_2 = 0$, we also prove statement~(iii).
Choose the notation such that $b_1 \geq b_2$. Here, $O_p( Z ( \tilde{C} ) ) = 
\langle t_1' \rangle \times \langle t_2' \rangle$ with $|t_j'| = p^a$ for 
$j = 1, 2$. Since $\det( t_1' )$ and $\det( t_2' )$ lie in the same subgroup of 
$\mathbb{F}^*$ and since $b_1 \geq b_2$, we get 
$|D| = |G \cap O_p( Z ( \tilde{C} ) )| = p^{a+b_2}$. On the other hand 
$|D| \geq |\tilde{D}|/p^a = p^{a+b_1+b_2}$, which implies that $b_1 = b_2 = 0$.
From $p \nmid n_1n_2$, we find $D = G \cap O_p( Z( \tilde{C} ) ) \cong
\langle t_1' \rangle$. This yields~(iii).

Let us finally prove the statements (i)--(vii).

We have already proved~(iii). Clearly, (v) follows from (i) and (ii). By~(iii), 
it suffices to prove~(iv) in the 
situation of~(i). In that case, $c = a_2 = 0$, so that $p \nmid |Z|$. Hence~$Y$ 
is trivial and so $D = \bar{D}$ is cyclic. Let us now prove~(i),~(ii).
In this situation, $|t| = p^{a+a_1}$ and $|D| = p^{a+a_1+a_2}$ 
by~(\ref{FirstDEstimate}). Moreover, $a_2 \leq c$ by~(\ref{cEstimate}).
Assume now that $a_1 > a_2$. As 
$n = m_1p^{a_1} + m_2p^{a_2}$ with $p \nmid m_1m_2$ and $a_1 > a_2$, we
get $b = a_2$. Since $c \leq b = a_2 \leq c$, we find $a_2 = c$.
It follows that
$$p^{a + a_1 - c'} = |\bar{D}| = |D/Y| \geq |D/O_p(Z)| = p^{a+a_1} 
\geq p^{a + a_1 - c'}.$$ 
Hence $c' = 0$, $Y = O_p(Z)$ and $\langle t \rangle \cap O_p( Z ) = \{ 1 \}$.
This proves~(i). Now assume that $a_1 = a_2 > 0$. Here, $|D| = p^{a+2a_1}$ 
by~(\ref{FirstDEstimate}). Since $|t| = p^{a+a_1}$ and 
$|\langle t \rangle \cap Y| = p^{c'}$ by definition, we get
$$p^{a + a_1 - c'} = |\bar{D}| = |D/Y|.$$ 
In particular, $|Y| = p^{a_1+c'}$. Moreover, 
$|\langle t \rangle \cap O_p( Z )| = p^{c-a_1}$, since 
$D = \langle t, O_p( Z ) \rangle$. Assume that 
$|Y| \leq p^{c-a_1}$. Then $Y \leq \langle t \rangle \cap O_p( Z )$, which 
implies $Y = \langle t \rangle \cap Y$. It follows that
$$p^{a_1 + c'} = |Y| = |\langle t \rangle \cap Y| = p^{c'},$$
which implies $a_1 = 0$, a contradiction. Hence $|Y| > p^{c-a_1}$. Then
$\langle t \rangle \cap O_p( Z ) \leq Y$, which implies 
$\langle t \rangle \cap Y = \langle t \rangle \cap O_p( Z )$. 
It follows
that $p^{c'} = p^{c - a_1}$, and so $c' = c -a_1$. Moreover,
$$|Y| = p^{a_1 + c'} = p^{a_1 + c - a_1} = p^c,$$
and thus $Y = O_p( Z )$. This completes the proof of~(ii).
To prove~(vi), assume that $n_1 = n_2$. Then $n = 2n_1$ and $p \nmid n_1$. Hence 
$c = 0$, contradicting $0 \leq c' < c$.
If $a_j > 0$ for $j = 1, 2$, then $|t_j| = p^{a+a_j}$ by 
Lemma~\ref{DeterminantLem}(a). If $a_1 > a_2 = 0$, then $|\det(t_1)| = p^a$ by 
Lemma~\ref{DeterminantLem}(d) and thus $|t_2| = p^a$, since $\det(t_1t_2) = 1$. 
If $a_1 = a_2 = 0$, then $t = t_1t_2 = t_2t_1$, $|t| = p^a$ and 
$\det(t_1t_2) = 1$ imply $|t_1| = |t_2| = p^a$. This proves~(vii).

(e) Since $r( \tilde{D} ) \leq 3$ by Lemma~\ref{TildeCmoduloC}(b), the factors
$\tilde{D}_j$ are cyclic for $1 \leq j \leq 3$. By 
\cite[Corollary~$3.6.2$]{HL25}, we have $|\tilde{D}_j| = p^{a + a_j + b_j}$ for 
$1 \leq j \leq 3$. The first two claims follow from~(\ref{EstimatingTildeD}).
For the consequences of these consult~(b) and Lemma~\ref{YEqualOpZ}(d) and~(e).

Let us now prove the final claim. We have $n = m_1p^{a_1} + n_2 + n_3$ and
$p \nmid n_2n_3$. Since $p^a \mid n$ and $a > 0$, we also have 
$p \mid n_2 + n_3$ since $a_1 > 0$. This implies $n_2 \neq n_3$ since $p$ is odd.
Now $\tilde{C} = C_{\tilde{G}}( t )$, and thus $N_{\tilde{G}}( \tilde{C} )$ 
permutes the spaces~$V_1$,~$V_2$ and~$V_3$. Since these have pairwise distinct
dimensions, $N_{\tilde{G}}( \tilde{C} )$ fixes all of these, which gives our 
claim.

(f) This is contained in (c)--(e).

(g) Suppose first that the minimal polynomial of~$t$ is irreducible. 
Then~$\tilde{D}$ is cyclic by~(c), and our claims hold.

Suppose next that the minimal polynomial of~$t$ has two irreducible 
factors. Then, $b_1 = b_2 = 0$ by~(d). Thus $|\tilde{D}| = p^{2a + a_1 + a_2}
= |O_p( Z( \tilde{C} ) )|$ and so $\tilde{D} = O_p( Z( \tilde{C} ) ) 
\leq Z( \tilde{C} )$. Hence 
$\tilde{C} \leq C_{\tilde{G}}( \tilde{D} )
\leq C_{\tilde{G}}( D ) = \tilde{C}$. Since $D = G \cap \tilde{D}$, we obtain
$N_{\tilde{G}}( \tilde{D} ) \leq N_{\tilde{G}}( D ) \leq 
N_{\tilde{G}}( C_{\tilde{G}}( D ) ) = N_{\tilde{G}}( \tilde{C} ) \leq
N_{\tilde{G}}( O_p( Z( \tilde{C} ) ) ) = N_{\tilde{G}}( \tilde{D} )$. 
The case when the minimal polynomial of~$t$ has three irreducible
factors is treated analogously.
\end{proof}

\addtocounter{subsection}{4}
\subsection{The intermediate configuration}
\label{IntermediateConfiguration}
Let us write $h \in \{ 1, 2, 3 \}$ for the number of irreducible factors of the
minimal polynomial of~$t$. If $h = 1$, let $\bar{E}$ denote the trivial subgroup
of~$\bar{D}$; otherwise, let~$\bar{E}$ be the unique subgroup of $\bar{D}$ of 
order~$p$.

Let~$\bar{\bd}$ denote a Brauer correspondent of~$\bar{\bB}$ 
in~$C_{\bar{G}}( \bar{E} )$. Then $W( \bar{\bB} ) = W( \bar{\bd} )$. The latter 
will be computed by pulling back $(\bar{E},\bar{\bd})$ to~$G$.

\addtocounter{thm}{1}
\begin{notn}
\label{HypoV}
{\rm
Assume the terminology of Notation~\ref{HypoII}. 

(i) Put $E := \langle t', Y \rangle$ with $t' = 1$, if $h = 1$, 
and $t' := t^{p^{a+a_1-c'-1}}$, otherwise. (As we have assumed that $\bar{D}$
is non-trivial and $|\bar{D}| = p^{a+a_1-c'}$ by 
Lemma~\ref{AnalysingTheDeltas}(f), we have $a+a_1-c' \geq 1$.)

(ii) Let $\tilde{\bH} := C_{\tilde{\bG}}( E ) = C_{\tilde{\bG}}( t' )$ and 
$\bH := \tilde{\bH} \cap \bG = C_{\bG}( E ) = C_{\bG}( t' )$.
}\hfill{\textsquare}
\end{notn}

Notice that $\tilde{\bH} = \tilde{\bG}$ and $\bH = \bG$ if the minimal 
polynomial of~$t$ is irreducible. Notice also that the inclusion 
$\bH \rightarrow \tilde{\bH}$ is a regular embedding, and that~$\tilde{\bH}$ 
and ~$\bH$ are $F$-stable.

We collect a few properties of the objects introduced above, assuming the
terminology of
{\rm Notations~\ref{HypoII}, \ref{HypoIII}, \ref{HypoIV}, \ref{HypoV}} and
{\rm Lemma~\ref{DeterminantLem}(e)}. Also, recall that 
$D = \langle t , O_p( Z ) \rangle$ by Lemma~\ref{YEqualOpZ}(c),
and that $|t| = p^{a+a_1}$ by Lemma~\ref{AnalysingTheDeltas}(f).

\begin{lem}
\label{PropertiesOfDPrime}
Assume that $h \geq 2$. For Parts~{\rm (e)} and~{\rm (f)} assume in addition 
that $|\{ n_1, n_2, n_3 \}| \geq 2$ if $h = 3$ and $a_1 = 0$. Then the 
following statements hold.

{\rm (a)} We have $|t'| = p^{c'+1}$. 

{\rm (b)} We have $N_{\tilde{G}}( D ) \leq N_{\tilde{G}}( E )$ and 
$N_{G}( D ) \leq N_{G}( E )$.

{\rm (c)} Suppose that $h = 2$. Then 
$\tilde{H} = \tilde{G}_1 \times \tilde{G}_2$ unless $a_1 = a_2 = 0$ and 
$c' < c$. In the latter case, $\tilde{H} = \tilde{G}$.

{\rm (d)} Suppose that $h = 3$. If $a_1 = 0$,
then $\tilde{H} = \tilde{G}_1 \times \tilde{G}_2 \times \tilde{G}_3$.

If $a_1 > 0$, then $\tilde{H} = \tilde{G}_1 \times \tilde{G}_{2,3}$ with 
$\tilde{G}_{2,3} := I( {V}_2 \oplus {V}_3, \kappa_{2,3} )$, where~$\kappa_{2,3}$ 
denotes the restriction of~$\kappa$ to ${V}_2 \oplus {V}_3$.

{\rm (e)} We have
$[N_G( E )\colon\!H] = [N_{\tilde{G}}( E )\colon\!\tilde{H}] 
\leq 2$.

{\rm (f)} We have $\bar{H} = C_{\bar{G}}( \bar{E} )$.
\end{lem}
\begin{proof}
An element $\xi \in \mathbb{F}^*$ is called rational, if 
$\xi \in \mathbb{F}_{q^{\delta}}$. If $\xi$ is a $p$-element, this is the case
if and only if $|\xi| \leq p^a$.

(a) This follows from $|t| = p^{a+a_1}$.

(b) Notice that~$Y$ is a central subgroup of~$\tilde{G}$ so that 
$Y \leq N_{\tilde{G}}( D )$ and $Y \leq N_{\tilde{G}}( E )$. Hence
$$N_{\tilde{G}}( D )/Y = N_{\tilde{G}/Y}( D/Y ) \leq 
N_{\tilde{G}/Y}( E/Y ) = N_{\tilde{G}}( E )/Y,$$
where the equalities arise from \cite[Lemma~$2.2$(a)]{HL24}, and the inclusion
is due to the fact that $D/Y$ is cyclic. Our assertions follow from this.

(c) Suppose first that $a_1 > a_2$, so that Lemma~\ref{AnalysingTheDeltas}(d)(i)
applies. As $c' = 0$, we have $|t'| = p \leq p^a$, so that the eigenvalues 
of~$t'$ on $V_j$ are rational for $j = 1, 2$. As 
$\langle t' \rangle \cap O_p( Z ) = \{ 1 \}$, the 
eigenvalues of~$t'$ on~$V_1$ and~$V_2$ are distinct. 

Now suppose that $a_1 = a_2$, so that 
Lemma~\ref{AnalysingTheDeltas}(d)(ii)(iii) applies. Hence $|D| = p^{a+2a_1}$. 
From $D = \langle t, O_p(Z) \rangle$ and $|t| = p^{a+a_1}$ we conclude that
$|\langle t \rangle \cap O_p(Z)| = p^{c-a_1}$. By~(a), we have
$t' \in Z \leq Z( \tilde{G} )$, if and only if $c' < c - a_1$. The latter can 
only happen if $a_1 > 0$.

Now suppose that $c' = c - a_1$. If $a_1 > 0$, then $c' + 1 = c - a_1 + 1
\leq c \leq a$, so that the eigenvalues of $t'$ on~$V_1$ and~$V_2$ are 
rational. If $a_1 = 0$, then $|t| = p^a$ and thus the
eigenvalues of $t'$ on on~$V_1$ and~$V_2$ are rational.

(d) Apply Lemma~\ref{AnalysingTheDeltas}(e). Suppose first that $a_1 > 0$. 
Then $|t_1| = p^{a+a_1}$ and $p^a \mid |t_j|$ for $j = 2, 3$. As 
$|t'| = p$ and $a + a_1 - 1 \geq a$, the eigenvalues of~$t'$ on~$V_1$ have 
order~$p$, and $V_2 \oplus V_3$ is the fixed space of~$t'$. This proves our 
assertion.

Suppose now that $a_1 = 0$. Then $|t| = p^a$ and all the eigenvalues of~$t$ are 
rational. Moreover, as $Y = O_p(Z)$ and $|O_p(Z)| = p^a$, we may assume that~$t$ 
acts as the identity on~$V_3$, i.e.\ $\xi_3 = 1$. From $\det(t) = 1$ and 
$p \nmid n_1n_2$ we conclude 
$|\xi_1| = |\xi_1^{n_1}| = |\xi_2^{n_2}| = |\xi_2|$. Thus
$|\xi_1| = |\xi_2| = p^a$.

The eigenvalues of~$t'$ are $\xi_1^{p^{a-1}}$, $\xi_2^{p^{a-1}}$ and~$1$.
Assume that $\xi_1^{p^{a-1}} = \xi_2^{p^{a-1}}$. Then $1 = \det(t') = 
\xi_1^{p^{a-1}n_1}\xi_2^{p^{a-1}n_2} = \xi_1^{p^{a-1}(n_1 + n_2)}$. As~$p$
divides $n = n_1 + n_2 + n_3$ and $p \nmid n_1n_2n_3$,
the sum $n_1+n_2$ is prime to~$p$. This contradicts the fact that 
$|\xi_1| = p^a$ and thus implies that $t'$ has three pairwise distinct 
eigenvalues. This proves our assertion.

(e) As $[\tilde{H}\colon\!H] = q - \varepsilon$ by
Lemma~\ref{CentralizersAndDeterminants}, we obtain $\tilde{H}G = \tilde{G}$.
It follows that $N_{\tilde{G}}( E )G = \tilde{G}$ and thus
$[N_{\tilde{G}}( E )\colon\!N_G( E )] = q - \varepsilon$.
Hence
$[N_G( E )\colon\!H] = [N_{\tilde{G}}( E )\colon\!\tilde{H}]$.
The structure of~$\tilde{H}$ determined in~(c) and~(d) implies that
$[N_{\tilde{G}}( E )\colon\!\tilde{H}] \leq 2$, as $N_{\tilde{G}}( E )$
permutes the eigen\-spac\-es of~$t'$ on~$V$. 

(f) This follows from~(e) and Lemma~\ref{CentralizersModZ}.
\end{proof}

The following example shows that the extra hypothesis for the statements~(e) 
and~(f) in Lemma~\ref{PropertiesOfDPrime} is necessary.

\begin{exmp}
{\rm Let $p = n = 3$, and consider $G = \SL_3( 7 )$. Then 
$\bar{G} := G/Z = \PSL_3( 7 )$ has a cyclic $3$-block~$\bar{\bB}$ of defect~$1$.
Now~$G$ has a unique conjugacy class of non-central elements of order~$3$. If we 
let 
$$
t := \left( 
       \begin{array}{ccc} \xi & 0 & 0 \\ 0 & \xi^{-1} & 0 \\ 0 & 0 & 1 \end{array}
       \right),
$$
where $\xi \in \mathbb{F}_7$ has order~$3$, we may assume that 
$D = \langle t, Z \rangle$ is a defect group of the block~$\bB$ of~$G$ 
dominating~$\bar{\bB}$.

We are thus in the situation that $h = 3$, $a_1 = a_2 = a_3 = 0$ and 
$n_1 = n_2 = n_3 = 1$. Moreover, $D = E$. However, $|C_G( E )| = 36 = 
|C_{\bar{G}}( \bar{E} )|$, so that $\overline{C_G( E )}$ has index~$3$
in~$C_{\bar{G}}( \bar{E} )$.
}\hfill{\textsquare}
\end{exmp}

\addtocounter{subsection}{3}
\subsection{The intermediate blocks}
Keep the notation of Subsection \ref{IntermediateConfiguration}. In 
particular,~$h$ denotes the number of irreducible factors of the minimal
polynomial of~$t$. Recall that 
$\bar{H} = \overline{C_{G}( E )} = C_{\bar{G}}( \bar{E} )$ by 
Lemma~\ref{PropertiesOfDPrime}(f), unless $h = 3$, $a_1 = 0$ and 
$n_1 = n_2 = n_3$. In any case, $\bar{H} \leq C_{\bar{G}}( \bar{E} )$. Notice 
that $C_G( D ) = C_H( D )$ and 
$C_{\tilde{G}}( \tilde{D} ) = C_{\tilde{H}}( \tilde{D} ) $.

\addtocounter{thm}{1}
\begin{notn} 
\label{HypoVI}
{\rm 
(i) Let $\bd$ denote the block of $H$ such that 
$( E, \bd ) \leq ( D, \bc )$ as Brauer pairs of~$H$.

(ii) Let $\tilde{\bd}$ denote the block of $\tilde{H}$ such that 
$( \tilde{D}, \tilde{\bc} )$ is a $\tilde{\bd}$-Brauer pair of~$\tilde{H}$.

(iii) Let $\bar{\bd}$ denote the block of~$\bar{H}$ dominated by~$\bd$.
}
\end{notn}

\begin{lem}
\label{IntermediateBlocksLem}
If $h = 1$, assume that $D \not\leq Z$. Then~$\tilde{\bd}$ covers~$\bd$ 
and~$\bar{\bd}$ is a Brauer correspondent of~$\bar{\bB}$.
\end{lem}
\begin{proof}
Observe that $C_{\tilde{H}}( D ) = C_{\tilde{G}}( D ) = \tilde{C}$. The 
block~$\tilde{\bc}$ of~$\tilde{C}$ has defect group~$\tilde{D}$, and 
so~$\tilde{\bd}$ has defect group~$\tilde{D}$ by Lemma~\ref{BrauerPairsLem}. 
Since $C_G( D ) = C_{H}( D )$, 
the inclusion $( E, \bd ) \leq ( D, \bc )$ of Brauer pairs of~$H$ is also an
inclusion of Brauer pairs of~$G$. Hence $( E, \bd )$ is a $\bB$-Brauer pair by 
\cite[Proposition~$6.3.6$]{LinckBook}, and so~$\bd$ has defect group~$D$, once
more by Lemma~\ref{BrauerPairsLem}.

Notice that $N_{\tilde{H}}( \tilde{ D } ) \leq N_{\tilde{H}}( D )$ by 
Lemma~\ref{AnalysingTheDeltas}(g). Let~$\tilde{\bb}$ and~$\bb$ denote the 
Brauer correspondent of~$\tilde{\bd}$, respectively~$\bd$, 
in~$N_{\tilde{H}}( D )$, respectively~$N_{H}( D )$.
By 
the Harris-Kn{\"o}rr correspondence~\cite[Theorem]{HK85}, it suffices to show 
that~$\tilde{\bb}$ covers~$\bb$. By definition of the inclusion of Brauer
pairs,~$\tilde{\bb}$ covers~$\tilde{\bc}$, and~$\bb$ covers~$\bc$.

As~$\tilde{\bc}$ covers~$\bc$ by definition,~$\tilde{\bb}$ covers some block
of~$N_{H}( D )$ covering~$\bc$. We will show that $N_H( D )/C$ is a $p$-group.
Then~$\bb$ is the unique block of~$N_{H}( D )$ covering~$\bc$ and 
hence~$\tilde{\bb}$ covers~$\bb$.

Suppose that $h = 1$, so that $\tilde{H} = \tilde{G}$. By
Lemma~\ref{AnalysingTheDeltas}(c) we have $C_{\tilde{G}}( D ) = \tilde{C} \cong
\GL_{n_1}^{\varepsilon}( q^{p^{a_1}} )$ and thus 
$|N_{\tilde{G}}( \tilde{C} )/\tilde{C}| = p^{a_1}$. Now consider the chain of
maps
$$N_G( D ) \hookrightarrow N_{\tilde{G}}( D ) \hookrightarrow 
N_{\tilde{G}}( C_{\tilde{G}}( D ) ) = N_{\tilde{G}}( \tilde{C} ) \rightarrow 
N_{\tilde{G}}( \tilde{C} )/\tilde{C},$$
whose kernel equals $N_G( D ) \cap \tilde{C} = C$. It follows that 
$N_G( D )/C$ is a $p$-group. 
Suppose now that $h \geq 2$. First, 
$\overline{N_{H}( D )} = N_{\bar{H}}( \bar{D} )$ by \cite[Lemma~$2.2$(a)]{HL24}. 
Thus 
$$N_{H}( D )/C \cong \overline{N_{H}( D )}/\bar{C} \cong 
N_{\bar{H}}( \bar{D} )/\bar{C}.$$
Now $N_{\bar{H}}( \bar{D} )/C_{\bar{H}}( \bar{D} )$ is a $p$-group,
since~$\bar{D}$ is cyclic and~$N_{\bar{H}}( \bar{D} )$ acts trivially 
on~$\bar{E}$. Also, $C_{\bar{H}}( \bar{D} )/\bar{C}$ is a $p$-group by
Lemma~\ref{CentralizersModZ}. 
Since 
$$|N_{\bar{H}}( \bar{D} )/\bar{C}| = 
|N_{\bar{H}}( \bar{D} )/C_{\bar{H}}( \bar{D} )| \cdot
|C_{\bar{H}}( \bar{D} )/\bar{C}|,$$ 
it follows that $N_{H}( D )/C$ is a $p$-group. 

We now show that~$\bar{\bd}$ is a Brauer correspondent of~$\bar{\bB}$. This is 
trivial if $h = 1$. Assume then that $h \geq 2$. As 
$\bar{H} = C_{\bar{G}}( \bar{D} )$, this statement makes sense. By 
Lemma~\ref{PropertiesOfDPrime}(b), we have $N_G( D ) \leq N_G( E )$. 
Let~$\bb'$ and~$\bar{\bb'}$ denote the Brauer correspondents of~$\bB$ 
in~$N_G( E )$, respectively of~$\bar{\bB}$ in $N_{\bar{G}}( \bar{E} )$.
Then $\overline{N_{G}( E )} = N_{\bar{G}}( \bar{E} )$ by 
\cite[Lemma~$2.2$(a)]{HL24}, and~$\bb'$ dominates~$\bar{\bb}'$ by
\cite[Lemma~$2.4.1$]{HL25}. 
Now $( E, \bd )$ is a $\bB$-Brauer pair
and so $\bb'$ covers~$\bd$. This easily implies that $\bar{\bb'}$ 
covers~$\bar{\bd}$, and so~$\bar{\bd}$ is a Brauer correspondent 
of~$\bar{\bB}$.
\end{proof}

\begin{figure}
\caption{\label{ExtendedDiagramForSLn} Some subgroups and blocks of $\GL^\varepsilon_n( q )$, II}
$$
\begin{xy}
\xymatrix@C+1pt{
\tilde{H} = C_{\tilde{G}}( E ), \tilde{\bd} \ar@{-}[rd] \ar@{-}[dd] & & \\ 
& H = C_G( E ), {\bd} \ar@{->>}[rd] \ar@{-}[dd] &  \\
\tilde{C} = C_{\tilde{G}}( D ), \tilde{\bc} \ar@{-}[rd] & & \bar{H}, \bar{\bd} \\ 
& C = C_G( D ), {\bc} & &
}
\end{xy}
$$
\end{figure}
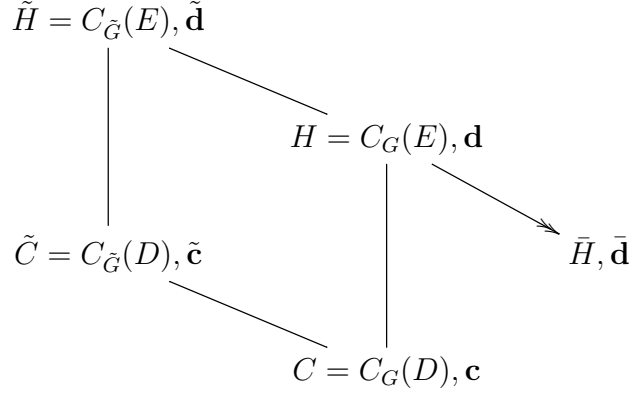

A diagram of the relevant groups and blocks is displayed in 
Figure~\ref{ExtendedDiagramForSLn}.

The maximal tori of $\tilde{C}$ and $\tilde{H}$ are also maximal tori 
of~$\tilde{G}$. The $\tilde{G}$-conjugacy classes of maximal tori of~$\tilde{G}$
are labelled by partitions of~$n$. If the $\tilde{G}$-conjugacy class of a 
maximal torus~$\tilde{T}$ is labelled by the partition 
$\pi = (f_1, f_2, \ldots , f_m)$, we call~$\pi$ the type of~$\tilde{T}$.
Then 
\begin{equation}
\label{ToriAndPartitions}
\tilde{T} \cong \tilde{T}_{f_1} \times \cdots \times \tilde{T}_{f_m}, 
\end{equation}
and
$N_{\tilde{G}}( \tilde{T} )$ fixes the sub-products of~(\ref{ToriAndPartitions})
corresponding to the same~$f_j$.

\begin{lem}
\label{TildeDIsStrictlyRegular}
If $h = 3$, assume that $a_1 > 0$. Then~$\tilde{\bd}$ is strictly regular with 
respect to a maximal torus $\tilde{T}$ of $\tilde{H}$ whose type has the parts 
$n_1, \ldots, n_h$.

In particular, there is $\theta \in \Irr( \tilde{T} )$ of $p'$-order and 
in general position with respect to~$\tilde{H}$, such that
$$\tilde{\chi} := \varepsilon_{\tilde{\bT}}\varepsilon_{\tilde{\bH}}
R_{\tilde{\bT}}^{\tilde{\bH}}( \theta )$$
is an irreducible character of~$\tilde{\bd}$.
\end{lem}
\begin{proof}
For $j \in \{ 1, \ldots , h \}$, the block $\tilde{\bc}_j$ of $\tilde{C}_j$ is
strictly regular with respect to a cyclic maximal torus $\tilde{T_j}$ of 
$\tilde{C}_j$, of type~$(n_j)$ when viewed as a torus of~$\tilde{G}_j$,
and~$\tilde{D}_j$ is a Sylow $p$-subgroup of~$\tilde{T_j}$; see 
\cite[Corollary~$3.6.2$]{HL25}. Thus $\tilde{T} = 
\tilde{T}_1 \times \cdots \times \tilde{T}_h$ contains $\tilde{D}$ as a
Sylow $p$-subgroup and $\tilde{\bc}$ is a strictly regular block of 
$\tilde{C} = \tilde{C}_1 \times \cdots \times \tilde{C}_h$ with respect 
to~$\tilde{T}$.

We claim that $N_{\tilde{H}}( \tilde{T} )$ fixes each of the factors 
$\tilde{T}_j$ for $1 \leq j \leq h$. This is trivial if $h = 1$. 
It is also clear if $h = 2$ and $\tilde{H} = \tilde{G}_1 \times \tilde{G}_2$.
If $h = 2$ and $\tilde{H} = \tilde{G}$, then $a_1 = a_2 = 0$ and $c' < c$ by
Lemma~\ref{PropertiesOfDPrime}(c). But then $n_1 \neq n_2$ by 
Lemma~\ref{AnalysingTheDeltas}(d)(vi), 
which also implies the claim. If $h = 3$, we have $a_1 > 0$ by assumption, and 
then $n_1 = m_1p^{a_1}, n_2, n_3$ are pairwise distinct. Indeed, $n_2 = n_3$ would 
give $n = m_1p^{a_1} + 2n_2$ and thus $p \mid n_2$, a contradiction. This yields 
our claim.

Now $\tilde{\bc}$ contains an irreducible character~$\tilde{\psi}$ of the form
$$\tilde{\psi} :=  
        \varepsilon_{\tilde{\bT}}\varepsilon_{\tilde{\bC}} 
        R_{\tilde{\bT}}^{\tilde{\bC}}( \theta ),$$
for some irreducible character~$\theta$ of~$\tilde{T}$ of $p'$-order and 
in general position with
respect to~$\tilde{C}$. By the claim,~$\theta$ is also in general position with 
respect to~$\tilde{H}$, and thus
$$\varepsilon_{\tilde{\bT}}\varepsilon_{\tilde{\bH}}
R_{\tilde{\bT}}^{\tilde{\bH}}( \theta ) = 
\varepsilon_{\tilde{\bC}}\varepsilon_{\tilde{\bH}}
R_{\tilde{\bC}}^{\tilde{\bH}}( \tilde{\psi} )$$
is an irreducible character of~$\tilde{H}$. Moreover,~$\tilde{\chi}$
lies in~$\tilde{\bd}$ by~\cite[Theorem~$2.5$]{CaEn99}. This proves our 
assertions.
\end{proof}

If $h = 1$, we have $\bar{\bB} = \bar{\bd}$, $\bB = \bd$, and 
$\tilde{\bB} = \tilde{\bd}$. In this case, we compute $W( \bar{\bB} )$ from 
$W( \bB )$ using \cite[Lemma~$2.4.2$(c)]{HL25}. If $h \geq 2$, then, by 
definition, $\bar{H} \leq C_{\bar{G}}( \bar{E} )$, where~$\bar{E}$ denotes the
unique subgroup of~$\bar{D}$ of order~$p$. In this case,
\cite[Remark~$2.3.3$]{HL25} shows that~$W( \bar{\bd} )$ can be 
computed from the sign sequence $\sigma_{\bar{\chi}}^{[l]}( \bar{t} )$, 
where~$\bar{\chi}$ denotes the non-exceptional character of~$\bar{\bd}$, and~$l$
is defined by $|\bar{D}| = p^l$. Namely, if $\Lambda = \{ 0, \ldots , l - 1 \}$
and $A \subseteq \Lambda \setminus \{ 0 \}$ is such that 
$\sigma_{\bar{\chi}}^{[l]}( \bar{t} ) = \omega_{\Lambda}( \mathbf{1}_A )$, then 
$W( \bar{\bd} ) = W_{\bar{D}}( A )$; for the notation see 
\cite[Definitions~$2.1.1$,~$2.1.2$]{HL25}.

\begin{lem}
\label{Approach}
If $h = 3$, assume that $a_1 > 0$.
Let $\bar{\chi}$ denote the non-exceptional character of~$\Irr(\bar{\bd})$,
and let~$\chi$ denote the inflation of~$\bar{\chi}$ to~$H$. Moreover, let
$\tilde{\chi} \in \Irr( \tilde{\bd} )$ be as in 
{\rm Lemma~\ref{TildeDIsStrictlyRegular}}.
Then
\begin{equation}
\label{ApproachEqI}
\sigma_{\chi}^{[a+a_1]}( t ) =
\omega_{\tilde{\bH}}^{[a+a_1]}( t ) = \sigma_{\tilde{\chi}}^{[a+a_1]}( t )
\end{equation}
and
\begin{equation}
\label{ApproachEqII}
\sigma_{\bar{\chi}}^{[a+a_1-c']}( \bar{t} ) =
\omega_{\tilde{\bH}}^{[a+a_1-c']}( t ).
\end{equation}
(For the notation $\omega_{\tilde{\bH}}$ see \cite[Definition~$2.5.6$]{HL25}).
\end{lem}
\begin{proof}
By definition,~$\bd$ dominates the nilpotent cyclic block~$\bar{\bd}$. By 
\cite[Lemma $2.3.4$]{HL25},
the character~$\bar{\chi}$ is the unique $p$-rational character in~$\bar{\bd}$.
As the irreducible characters of~$\bd$ which do not have $Y$ in their kernels
are not $p$-rational, $\Irr( \bd )$ has a unique $p$-rational character,
namely~$\chi$, and~$\chi$ lifts the unique irreducible Brauer character 
of~$\bd$. Thus the hypotheses of \cite[Lemma $2.5.18$]{HL25} are satisfied, and 
so~$\tilde{\chi}$ lies above~$\chi$.

As $\bar{\chi}$ is the non-exceptional character of $\Irr(\bar{\bd})$, the
values $\bar{\chi}( \bar{u} )$ are non-zero integers for $\bar{u} \in \bar{D}$;
see \cite[Lemma~$3.3$]{HL24}. The same is then true for the values $\chi( u )$
for $u \in \langle t \rangle$.

By \cite[Lemma $2.5.18$]{HL25} we have
$\sigma_{\tilde{\chi}}( u ) = \sigma_{\chi}(u)$ for all $u \in \tilde{D}$, which
implies
$$\sigma_{\tilde{\chi}}^{[a+a_1]}( t ) = 
\sigma_{\chi}^{[a+a_1]}( t ).$$
By \cite[Lemma~$2.5.7$]{HL25},
$$\sigma_{\tilde{\chi}}^{[a+a_1]}( t ) = 
\omega_{\tilde{\bH}}^{[a+a_1]}( t ),$$
which yields~(\ref{ApproachEqI}). This trivially implies
$$\sigma_{\chi}^{[a+a_1 - c']}( t ) =
\omega_{\tilde{\bH}}^{[a+a_1 - c']}( t ).$$
Since~$\chi$ is the inflation of $\bar{\chi}$ to~$H$, and as the kernel
of the epimorphism $\langle t \rangle \rightarrow \langle \bar{t} \rangle 
= \bar{D}$ has order~$p^{c'}$, we obtain
$$\sigma_{\bar{\chi}}^{[a+a_1-c']}( \bar{t} ) = 
\sigma_{\chi}^{[a+a_1-c']}( t ),$$
which gives~(\ref{ApproachEqII}).
\end{proof}

\addtocounter{subsection}{4}
\subsection{Computing the invariants in case $h = 1$}
Assume that the minimal polynomial of~$t$ is irreducible. Recall the results of 
Lemmas \ref{AnalysingTheDeltas}(c) and~\ref{IntermediateBlocksLem} in this case. 
If $D \not \leq Z$, then $n = m_1p^{a + a_1}$ with $p \nmid m_1$ and $a_1 > 0$. 
The defect group~$D$ is cyclic of order $p^{a+a_1}$, and~$\bB$ is covered by a 
cyclic block $\tilde{\bB}$ of~$\tilde{G}$ with defect group~$\tilde{D}$ of 
order~$p^{2a+a_1}$. Moreover, $|\bar{D}| = p^{a + a_1 - c'}$.

\addtocounter{thm}{1}
\begin{prop}
\label{Case1Prop}
Suppose that the minimal polynomial of~$t$ is irreducible. If $D \leq Z$, then
$W( \bar{\bB} ) \cong k$.

Assume in the following that $D \not \leq Z$, so that $a_1 \geq 1$ by
{\rm Lema~\ref{AnalysingTheDeltas}(c)}. Then $W( \bar{\bB} ) \cong k$
if $\varepsilon = 1$, or if $p \equiv 1\,\,(\mbox{\rm mod}\,\,4)$, or if~$m_1'$ 
is even, or if $c' = a$ and $a_1 = 1$. Otherwise, 
$W( \bar{\bB} ) = W_{\bar{D}}( [a - c', a + a_1 - c' - 1] )$, if $c' < a$,
and $W( \bar{\bB} ) = W_{\bar{D}}( [1, a_1 - 1] )$, if $c' = a$ and $a_1 > 1$.
\end{prop}
\begin{proof}
If $D \leq Z$, then $\bar{D} \leq Z( \bar{G} )$, and so
$W( \bar{\bB} ) \cong k$ by \cite[Lemma~$3.6$(b)]{HL24}. Assume then that
$D \not\leq Z$.

Put $l := a + a_1$ and $\Lambda := \{ 0, \ldots , l - 1\}$. Since $t^{p^{l -1}}$ 
has order~$p$, we have $t^{p^{l - 1}} \in Z( \tilde{H} )$. Let~$\tilde{\chi}$ be 
as in Lemma~\ref{TildeDIsStrictlyRegular}. Then \cite[Lemma~$4.1.3$]{HL25} 
yields $\sigma_{\tilde{\chi}}^{[a + a_1]}( t ) = \omega_{\Lambda}( I )$ with 
$I = \emptyset$, unless $\varepsilon = -1$,~$n$ is odd, and 
$p \equiv -1\,\,(\mbox{\rm mod}\,\,4)$, in which case $I = [a, l - 1]$. In view 
of \cite[Remark~$2.3.3$]{HL25} and 
Equation~(\ref{ApproachEqI}) in Lemma~\ref{Approach}, this gives 
$W( \bB ) \cong W_{D}( I )$.

Our assertion follows from this and \cite[Lemma~$2.4.2$(c)]{HL25}.
\end{proof}

\addtocounter{subsection}{1}
\subsection{Computing the invariants in case $h = 2$}
We now assume that the minimal polynomial of $t$ has exactly two irreducible 
factors. Recall the following facts from Lemma~\ref{AnalysingTheDeltas}(d). We
have 
$n = n_1 + n_2$ with $n_1 = m_1p^{a_1}$, $n_2 = m_2p^{a_2}$ for some 
non-negative integers $a_1 \geq a_2$, and $p \nmid m_1m_2$. Also, $c' = 0$ if 
$a_1 > a_2$, and $c' = c - a_1$ if $a_1 = a_2 > 0$. Moreover, $|\bar{D}| = p^l$ 
with $l = a + a_1 - c'$.

Recall that $V = V_1 \oplus V_2$ is the primary decomposition of~$V$ with 
respect to~$t$, where $\dim(V_j) = n_j = m_jp^{a_j}$ for $j = 1, 2$. Recall also 
from 
Notation~\ref{HypoI}(ii) that $\bV = \mathbb{F}^n$ is the natural vector space 
for~$\tilde{\bG}$. For $i = 1, 2$, let~$\bV_i$ denote the $\mathbb{F}$-span 
of~$V_i$, and let $\tilde{\bG}_i$ denote the subgroup of~$\tilde{\bG}$ induced
on~$\bV_i$. Then~$\tilde{\bG}_i$ is $F$-stable and
$\tilde{\bG}_i^F = \tilde{G}_i$ for $i = 1, 2$.

\addtocounter{thm}{1}
\begin{prop}
\label{MainCorCase2}
Suppose that the minimal polynomial of~$t$ has exactly two irreducible factors. 

Then $W( \bar{\bB} ) = W_{\bar{D}}( I )$, with 
$I \subseteq \{ 1, \ldots, l - 1 \}$ an interval. 
If $l = 1$, or if $\varepsilon = 1$, or if $a_1 = a_2 = 0$ 
and $c' = c$, then $I = \emptyset$, i.e.\ $W( \bar{\bB} ) \cong k$. 

Suppose in the following that $l \geq 2$, that $\varepsilon = -1$, and that 
$c' < c$ if $a_1 = a_2 = 0$. Then~$I$ is non-empty exactly in the following 
cases.

{\rm (a)} We have $a_1 > a_2$, $p \equiv -1\,\,(\mbox{\rm mod}\,\,4)$ and at 
least one of~$n_1$,~$n_2$ is odd. Then
$$I = 
\begin{cases} 
[a, l - 1], & \text{\ if\ } n_1 \text{\ odd and\ } (n_2 \text{\ even or\ } a_2 = 0); \\
[a + a_1 - a_2, l - 1], & \text{\ if\ } n_1 \text{\ even,\ } n_2 \text{\ odd, and\ } a_2 > 0; \\
[a, l - a_2 - 1], & \text{\ if\ } n_1 \text{\ and\ } n_2 \text{\ odd}.
\end{cases}
$$

{\rm (b)} We have $a_1 = a_2 > 0$, $p \equiv -1\,\,(\mbox{\rm mod}\,\,4)$, and~$n$ 
is odd. Then
$$I = [a - c', l - 1].
$$

{\rm (c)} We have $a_1 = a_2 = 0$, $c' < c$ and $n_1, n_2$ odd. Then
$$I = \{ c - c' \}.$$
\end{prop}
\begin{proof}
If $l = 1$, then $W( \bar{\bB} ) \cong k$ by \cite[Lemma~$3.6$(b)]{HL24}.
Suppose now that $a_1 = a_2 = 0$ and that $c' = c$. Then 
$\tilde{H} = \tilde{G}_1 \times \tilde{G}_2$ by
Lemma~\ref{PropertiesOfDPrime}(c). On the other hand, $\tilde{C} = 
\tilde{C}_{1} \times \tilde{C}_{2} = \tilde{G}_1 \times \tilde{G}_2$ since 
$a_1 = a_2 = 0$. It follows that $\tilde{H} = \tilde{C}$ and thus $H = C$. 
This implies $C_{\bar{G}}( \bar{E} ) = \bar{H} = 
\bar{C} \leq C_{\bar{G}}( \bar{D} ) \leq C_{\bar{G}}( \bar{E} )$, the first 
equality arising from Lemma~\ref{PropertiesOfDPrime}(f).
Our claim follows from \cite[Lemma~$3.6$(b)]{HL24}.

To continue, assume that $l \geq 2$ and that $c' < c$ if $a_1 = a_2 = 0$. 
Set $\Lambda := \{ 0, 1, \ldots , l - 1 \}$. To determine
$\omega_{\tilde{\bH}}^{[l]}( t )$, which yields $W( \bar{\bB} )$ by 
Lemma~\ref{Approach} and the remarks preceding it, we first determine 
$\omega_{\tilde{\bG}_j}^{[l]}( t_j )$ for $j = 1, 2$. Recall that 
$|t_j| = p^{a+a_j}$ for $j = 1, 2$; see Lemma~\ref{AnalysingTheDeltas}(d)(vii). 
Thus
$|t_2^{p^{l-1}}| \leq |t_1^{p^{l-1}}| = p^{c' + 1} \leq p^a$, where the last
inequality arises from Lemma~\ref{AnalysingTheDeltas}(d)(i)(ii), respectively 
our hypothesis $c' < c$ if $a_1 = a_2 = 0$. Hence
$t_j^{p^{l-1}} \in Z( \tilde{G}_j )$ for $j = 1, 2$, and 
\cite[Lemmas~$2.5.7$,~$4.1.3$]{HL25} yield
$$\omega_{\tilde{\bG}_j}^{[l]}( t_j ) = \omega_{\Lambda}( \mathbf{1}_{I_j} )$$
with intervals $I_j \subseteq \Lambda \setminus \{ 0 \}$, determined as follows.
If~$n_j$ is odd, $\varepsilon = -1$ and $p \equiv -1\,\,(\mbox{\rm mod}\,\,4)$,
then $$I_j = [a + a_1 - c' - a_j, l - 1];$$
otherwise $I_j = \emptyset$. Notice that $I_2 = \emptyset$ if $a_2 = 0$.

If $a_1 > a_2$ or if $a_1 = a_2$ and $c' = c - a_1$, we have
$\tilde{\bH} = \tilde{\bG}_1 \times \tilde{\bG}_2$ by
Lemma~\ref{PropertiesOfDPrime}(c), and thus 
\begin{equation}
\label{FactorizationOfOmega}
\omega_{\tilde{\bH}}^{[l]}( t ) = 
\omega_{\tilde{\bG}_1}^{[l]}( t_1 ) \omega_{\tilde{\bG}_2}^{[l]}( t_2 ),
\end{equation}
where the two $l$-tuples on the right hand side are multiplied component-wise.

We continue the proof under the assumption that~(\ref{FactorizationOfOmega})
holds. Let $I := I_1 \diamond I_2$ denote the symmetric difference of $I_1$ 
and~$I_2$.
Then $\omega_{\tilde{\bG}_1}^{[l]}( t_1 )\omega_{\tilde{\bG}_2}^{[l]}( t_2 )
= \omega_{\Lambda}( I )$; see \cite[Subsection~$2.2$]{HL25}.
This yields the assertions in~(a) and~(b).

Now assume that~(\ref{FactorizationOfOmega}) is not satisfied. Then
$a_1 = a_2 = 0$ and $c' < c$. In particular, $D = \langle t \rangle$ and 
$|t| = p^a$ by Lemma~\ref{AnalysingTheDeltas}(d)(iii). Put $l_1 := c - c'$ and 
$l_2 := l - l_1$. The $l$-tuple
$\omega_{\tilde{\bH}}^{[l]}( t )$ contains the values of $\omega_{\tilde{\bH}}$
at the elements $t^{p^{l-1}}, t^{p^{l-2}}, \ldots , t^p, t$. As~$j$ runs 
from~$1$ to~$l$, the order of $t^{p^{l-j}}$ runs from $p^{c' + 1}$ to $p^{a}$. 
For $j = l_1$ we get $|t^{p^{l-j}}| = p^{c}$.
As $O_p( Z ) \leq D = \langle t \rangle$, we have
$t^{p^{l-j}} \in Z$ exactly for $j = 1, \ldots , l_1$. 
Thus 
\begin{equation}
\label{FactorizationOfCentralizer}
C_{\tilde{\bH}}( t^{p^{l-j}} ) = 
\begin{cases}
{\tilde{\bH}}, & \text{\rm for\ } 1 \leq j \leq l_1; \\
C_{\tilde{\bG}_1}( t_1^{p^{l-j}} ) \times C_{\tilde{\bG}_2}( t_2^{p^{l-j}} ), & 
\text{\rm for\ } l_1 + 1 \leq j \leq l.
\end{cases}
\end{equation}
By Lemma~\ref{AnalysingTheDeltas}(c) we have $\tilde{\bH} = \tilde{\bG}$. Hence
\begin{equation}
\label{PartialFactorizationOfOmega}
\omega_{\tilde{\bH}}( t^{p^{l-j}} ) = 
\begin{cases}
1, & \text{\rm for\ } 1 \leq j \leq l_1; \\
\varepsilon_{\tilde{\bG}} \varepsilon_{\tilde{\bG}_1} \varepsilon_{\tilde{\bG}_2}
\omega_{\tilde{\bG}_1}( t_1^{p^{l-j}} ) \omega_{\tilde{\bG}_2}( t_2^{p^{l-j}} ),
& \text{\rm for\ } l_1 + 1 \leq j \leq l.
\end{cases}
\end{equation}
It follows that $\omega_{\tilde{\bH}}^{[l]}( t )$ is the concatenation of the 
all-$1$-vector $(1, \ldots , 1)$ of length $l_1$ with
$$\varepsilon_{\tilde{\bG}} \varepsilon_{\tilde{\bG}_1} \varepsilon_{\tilde{\bG}_2}
\omega_{\tilde{\bG}_1}^{[l_2]}( t_1 ) \omega_{\tilde{\bG}_2}^{[l_2]}( t_2 ).$$
By \cite[Example~$2.5.5$]{HL25}, we have
$$\varepsilon_{\tilde{\bG}} \varepsilon_{\tilde{\bG}_1} \varepsilon_{\tilde{\bG}_2} =
\begin{cases}
-1, & \text{\rm if}\  \varepsilon = -1 \text{\rm\ and\ } n_1, n_2 \text{\rm\ odd}; \\
 1, & \text{\rm otherwise}.
\end{cases}
$$
Suppose in addition that $\varepsilon = -1$ and that $n_1$ and $n_2$ are odd.
Then $I_1 = I_2$, and so 
$\omega_{\tilde{\bG}_1}( t_1^{p^{l-j}} ) \omega_{\tilde{\bG}_2}( t_2^{p^{l-j}} ) = 1$
for all $1 \leq j \leq l$.
Thus 
$$\omega_{\tilde{\bH}}^{[l]}( t ) = (1, \ldots , 1, -1, \ldots , -1),$$
where the first entry $-1$ is at position $l_1 + 1$. Hence
$\omega_{\tilde{\bH}}^{[l]}( t ) = \omega_{\Lambda}( \{ l_1 \} )$
by \cite[Lemma~$2.2.1$]{HL25}. This yields the instance listed in~(c).
If $\varepsilon = 1$ or $\varepsilon = -1$ and at least one of $n_1$,~$n_2$ is even, then 
$\varepsilon_{\tilde{\bG}} \varepsilon_{\tilde{\bG}_1} \varepsilon_{\tilde{\bG}_2} = 1$,
and thus~(\ref{FactorizationOfOmega}) holds by~(\ref{FactorizationOfCentralizer})
and~(\ref{PartialFactorizationOfOmega}), and the fact that
$t_i^{p^{l-j}} \in Z( \tilde{G}_i )$ for $i = 1, 2$ and all $1 \leq j \leq l_1$.
This contradiction concludes our proof.
\end{proof}

\addtocounter{subsection}{1}
\subsection{Computing the invariants in case $h = 3$}
We now assume that the minimal polynomial of $t$ has exactly three irreducible
factors. Recall the following facts from Lemma~\ref{AnalysingTheDeltas}(e). We 
have $n = n_1 + n_2 + n_3$ with $n_1 = m_1p^{a_1}$ 
for a non-negative integer~$a_1$, and $p \nmid m_1n_2n_3$. Moreover, $c' = 0$, 
and $|\bar{D}| = p^l$ with $l = a + a_1$. Notice that $l \geq 2$ if $a_1 > 0$.

Recall that $V = V_1 \oplus V_2 \oplus V_3$ is the primary decomposition of~$V$ 
with respect to~$t$, where $\dim(V_j) = n_j$ for $j = 1, 2, 3$. Recall also from
Notation~\ref{HypoI}(ii) that $\bV = \mathbb{F}^n$ is the natural vector space 
for~$\tilde{\bG}$. For $i = 1, 2, 3$, let~$\bV_i$ denote the $\mathbb{F}$-span 
of~$V_i$, and let~$\tilde{\bG}_i$ denote the subgroup of~$\tilde{\bG}$ induced 
on~$\bV_i$. Put $\bV_{2,3} := \bV_2 \oplus \bV_3$, and let $\tilde{\bG}_{2,3}$ 
denote the subgroup of~$\tilde{\bG}$ induced on~$\bV_{2,3}$. 
Then~$\tilde{\bG}_i$ is $F$-stable, and $\tilde{\bG}_i^F = \tilde{G}_i$ for 
$i = 1, 2, 3$. Similarly, $\tilde{\bG}_{2,3}$ is $F$-stable and 
$\tilde{\bG}_{2,3}^F = \tilde{G}_{2,3}$ in the notation of 
Lemma~\ref{PropertiesOfDPrime}(c)(iii). 

\addtocounter{thm}{1}
\begin{prop}
\label{MainCorCase3}
Suppose that the minimal polynomial of~$t$ has exactly three irreducible 
factors. Then $W( \bar{\bB} ) = W_{D}( A )$, with
$A \subseteq \{ 1, \ldots, l - 1 \}$ a union of two intervals. If $a_1 = 0$ or 
if $\varepsilon = 1$, then $A = \emptyset$, i.e.\ $W( \bar{\bB} ) \cong k$. 

Suppose in the following that $a_1 > 0$ and $\varepsilon = -1$. Then~$A$ is
non-empty exactly in the following cases.

{\rm (a)} At least one of $n_2$, $n_3$ is even, $n_1$ is odd and 
$p \equiv -1\,\,(\mbox{\rm mod}\,\,4)$. Then
$$A = [a, l - 1].$$

{\rm (b)} Each of $n_2$ and $n_3$ is odd and $a_1 \geq a$. Then

$$A = 
\begin{cases}
\{ a_1 \}, & \text{if\ } n_1 \text{\ is even or\ } p \equiv 1\,\,(\mbox{\rm mod}\,\,4); \\
[a, l - 1] \setminus \{ a_1 \}, & 
\text{if\ } n_1 \text{\ is odd and\ } p \equiv -1\,\,(\mbox{\rm mod}\,\,4).
\end{cases}
$$

{\rm (c)} Each of $n_2$ and $n_3$ is odd and $a_1 < a$. Then
$$A = 
\begin{cases}
\{ a_1 \}, & \text{if\ } n_1 \text{\ is even or\ } p \equiv 1\,\,(\mbox{\rm mod}\,\,4); \\
\{ a_1 \} \cup [a, l - 1], & \text{if\ } n_1 \text{\ is odd and\ } p \equiv -1\,\,(\mbox{\rm mod}\,\,4).
\end{cases}
$$
\end{prop}
\begin{proof}
Suppose first that $a_1 = 0$ and $n_1 = n_2 = n_3$. Then $n = 3n_1$, and thus 
$p = 3$ and $a = 1$, as $p \nmid n_1$. In this case, $|\bar{D}| = 3$, hence 
$W( \bar{\bB} ) \cong k$ by \cite[Lemma~$3.6$(b)]{HL24}.

Suppose next that $a_1 = 0$ and $|\{ n_1, n_2, n_3 \}| \geq 2$. Then
$\tilde{H} = \tilde{G}_1 \times \tilde{G}_2 \times \tilde{G}_3$ by
Lemma~\ref{PropertiesOfDPrime}(d). Thus $\tilde{H} = \tilde{C}$ and
$H = C = C_G( D )$. In particular, $\bar{H} \leq C_{\bar{G}}( \bar{D} )$.
Now $\bar{H} = C_{\bar{G}}( \bar{E} )$ by Lemma~\ref{PropertiesOfDPrime}(f), and 
so $\bar{H} \leq C_{\bar{G}}( \bar{D} ) \leq C_{\bar{G}}( \bar{E} ) = \bar{H}$. 
Our assertion follows from \cite[Lemma~$3.6$(b)]{HL24}.

Suppose from now on that $a_1 > 0$. Set $t_{2,3} := t_2t_3$ and 
$\Lambda := \{ 0, 1, \ldots , l - 1 \}$. Since $C_{\tilde{\bH}}( t )
= C_{\tilde{\bG}_1}( t_1 ) \times C_{\tilde{\bG}_{2,3}}( t_{2,3} )$ by
Lemma~\ref{PropertiesOfDPrime}(d), we obtain 
$$\omega_{\tilde{\bH}}^{[l]}( t ) = 
\omega_{\tilde{\bG}_1}^{[l]}( t_1 ) 
\omega_{\tilde{\bG}_{2,3}}^{[l]}( t_{2,3} ),$$
where the multiplication of the $l$-tuples on the right hand side is defined
component-wise. 

By \cite[Corollary~$4.1.4$]{HL25}
we have
$$\omega_{\tilde{\bG}_1}^{[l]}( t_1 ) = \omega_{\Lambda}( \mathbf{1}_{I_1} )$$
with $I_1 = \emptyset$ unless $n_1$ is odd, $\varepsilon = -1$ and
$p \equiv -1\,\,(\mbox{\rm mod}\,\,4)$, in which case $I_1 = [a, l - 1]$.

Let $1 \leq j \leq l$ and put $u := t_{2,3}^{p^{l-j}}$.
Lemma~\ref{AnalysingTheDeltas}(e) implies that
$C_{\tilde{\bG}_{2,3}}( u ) = \tilde{\bG}_{2,3}$ if $j \leq a_1$, and
$C_{\tilde{\bG}_{2,3}}( u ) = \tilde{\bG}_2 \times \tilde{\bG}_3$ if
$a_1 + 1 \leq j \leq a + a_1$. In the former case, 
$\omega_{\tilde{\bG}_{2,3}}( u ) = 1$. In the latter case,
$$\omega_{\tilde{\bG}_{2,3}}( u ) = 
\varepsilon_{\tilde{\bG}_{2,3}}\varepsilon_{\tilde{\bG}_{2}}\varepsilon_{\tilde{\bG}_{3}}.$$
By \cite[Example~$2.5.5$]{HL25}, we get
$$
\varepsilon_{\tilde{\bG}_{2,3}}\varepsilon_{\tilde{\bG}_{2}}\varepsilon_{\tilde{\bG}_{3}} =
\begin{cases}
-1, & \text{\rm if}\  \varepsilon = -1 \text{\rm\ and\ } n_2, n_3 \text{\rm\ odd}; \\
 1, & \text{\rm otherwise}.
\end{cases}
$$
It follows that $\omega_{\tilde{\bG}_{2,3}}^{[l]}( t_{2,3} ) = 
\omega_{\Lambda}( \mathbf{1}_{I_{2,3}} )$ with $I_{2,3} = \{ a_1 \}$,
if $n_2$ and~$n_3$ are odd and $\varepsilon = -1$, and $I_{2,3} = \emptyset$, 
otherwise.

We conclude from the considerations in \cite[Subsection~$2.2$]{HL25} that
$$\omega_{\tilde{\bG}_1}^{[l]}( t_1 ) \omega_{\tilde{\bG}_{2,3}}^{[l]}( t_{2,3} ) =
\omega_{\Lambda}^{[l]}( \mathbf{1}_{A} )$$ 
with $A = I_1 \diamond I_{2,3}$, 
the symmetric difference of $I_1$ and~$I_{2,3}$. This yields our assertions.
\end{proof}

We record a specific corollary of the above results. 
If $\tilde{T}$ is a maximal torus of~$\tilde{G}$ of type~$\pi$,
we write $T_{\pi} := \tilde{T} \cap G$. As every maximal torus of
$\SL_n^\varepsilon(q)$ is of this form, this yields a labelling of
the maximal tori of $\SL_n^\varepsilon(q)$ (up to conjugation) by 
partitions of~$n$.

\begin{cor}
\label{SpecialCases}
Suppose that $p = 3$ and $n \in \{ 3, 6, 9 \}$ or that $p = 5$ and $n = 5$. 
Let~$\bB$ be a $p$-block of 
$G = \SL_n^\varepsilon( q )$ with a non-trivial cyclic defect group~$D$. 
If $n = 9$ assume that $D \not\leq Z$ and that $a \geq 2$. 
Then the following statements hold.

{\rm (a)} We have $|D| = p$ or $|D| = p^a$. Also, 
$W( \bB ) \cong k$ or $W( \bB ) \cong W_D( \{ 1 \} )$. The latter occurs
exactly if $\varepsilon = -1$, $|D| = 3^a$ with $a \geq 2$, and $n = 6$
In this case, we have $h = 2$ and $\{ n_1 , n_2 \} = \{ 5, 1 \}$
in the notation of {\rm Proposition~\ref{MainCorCase2}}.

{\rm (b)} Suppose that $p = 3$ and $n = 6$. If $W( \bB ) \not\cong k$, 
then~$\bB$ is strictly regular with respect to~$T_{(5,1)}$.
\end{cor}
\begin{proof}
(a) In the setup of Notation~\ref{HypoII} we have $Y = \{ 1 \}$, hence
$\bar{\bB} = \bB$ and  $c' = 0$. Also, the group~$D$ of 
Notation~\ref{HypoII}(ii) is cyclic, a defect group of~$\bB$, and 
$|D| = |t| = p^{a+a_1}$, the latter by Lemma~\ref{AnalysingTheDeltas}(f). 
In particular, we cannot have $h = 3$, since in that case $Y = O_p( Z ) 
\neq \{ 1 \}$.

If $D \leq Z$, then $W( \bB ) \cong k$ by \cite[Lemma~$3.6$(b)]{HL24}. 
Moreover, $n \neq 9$ by hypothesis, and hence $|D| = p$. Assume then 
that~$D$ is non-central in the following. Suppose that $h = 1$. Then 
$p^{a+a_1} \mid n$ with $a_1 \geq 1$ by Lemma~\ref{AnalysingTheDeltas}(c).
As we have assumed that $a \geq 2$ if $n = 9$, this case cannot occur.

Suppose finally that $h = 2$. As~$D$ is cyclic, $c > 0$, and $|D| = p^{a+a_1}$, 
we must have $a_1 = a_2 = 0$ by Lemma~\ref{AnalysingTheDeltas}(d)(i)(ii). In
particular, $|D| = p^a$. Proposition~\ref{MainCorCase2} shows that 
$W( \bB ) \cong k$, unless $\varepsilon = -1$, $a \geq 2$, and~$n_1$ and~$n_2$ 
are odd.  If all these latter conditions are satisfied, $n = 6$ and 
$\{ n_1, n_2 \} = \{ 5, 1 \}$ since $3 \nmid n_1n_2$. Moreover,
$W( \bB ) \cong W_D( \{ c - c' \} ) = W_D( \{ 1 \} )$ by 
Proposition~\ref{MainCorCase2}(c).

(b) Suppose that $p = 3$, $n = 6$ and $W( \bB ) \not\cong k$. Without loss of
generality we may assume that $n_1 = 5$ and $n_2 = 1$. In the notation of
Lemma~\ref{DeterminantLem}(e), we have
$\tilde{C}_1 = \tilde{G}_1 = \GU_5( q )$ and
$\tilde{C}_2 = \tilde{G}_2 = \GU_1( q )$. Moreover, $D \leq \tilde{D} =
\tilde{D}_1 \times \tilde{D}_2$, where $\tilde{D}_j$ is a cyclic defect group of
a block of $\tilde{C}_j$, $j = 1, 2$. In 
particular,~$\tilde{D}$ is a Sylow $3$-subgroup of a maximal torus 
$\tilde{T} \leq \tilde{G}$ of type $(5,1)$; see \cite[Corollary~$3.6.2$]{HL25}. 

Let $\chi \in \Irr( \bB )$. Then~$\chi$ has height~$0$, and thus $\chi(1)_3 = 
3^{4a + 2}$, as $|G|_3 = 3^{5a+2}$. The character degrees of~$G$ given 
in~\cite{LL2} show that $\chi(1) = [G\colon\!T]_{r'}/f$ with $T = T_{(4,1,1)}$ 
and $f \in \{1,2\}$, or $T = T_{(5,1)}$ and $f = 1$. Hence~$\bB$ is regular
with respect to one of these tori. The former case cannot occur, as the torus 
$T_{(4,1,1)}$ does not have a cyclic Sylow $3$-subgroup. Hence~$\bB$ is strictly
regular with respect to~$T_{(5,1)}$.
\end{proof}

\begin{rem}
\label{SpecialCasesRem}
%
Suppose that $p = 3$ and $n = 9$, and that $a \geq 2$, so that $c = 2$.
Let $Y \leq Z$ with $|Y| = 3$, and let~$\bar{\bB}$ be a $3$-block of 
$\bar{G} = G/Y$ with a non-central, cyclic defect group~$\bar{D} = D/Y$. Then 
$h = 2$ and $a_1 = a_2 = 0$ by Lemma~\ref{AnalysingTheDeltas}(c)(d)(e). 

Let~$\bB$ be a block of~$G$ with defect group~$D$ dominating~$\bar{\bB}$. Then 
$D \not\leq Z$ by hypothesis on~$\bar{D}$. In particular,~$\bB$ is as in 
Corollary~\ref{SpecialCases}, so that~$D$ is cyclic of order~$3^a$. Thus 
$c' = 1$ and $|\bar{D}| = 3^{a-1}$. By Corollary~\ref{SpecialCases}(a) and 
\cite[Lemma~$2.4.2$(c)]{HL25}, we have $W( \bar{\bB} ) \cong k$. 
%
\end{rem}

\section{Synthesis}\label{sec:Synthesis}
\label{Existence} 
We now investigate which of the parameter sets exhibited in 
Section~\ref{TheSpecialLinearAndUnitaryGroups} correspond to blocks.
Throughout this section, we fix an odd prime~$p$.

Recall that we have fixed a sign $\varepsilon \in \{ -1, 1\}$. As in
Notation~\ref{HypoI}(iv), we let $\delta = 1$, if $\varepsilon = 1$, and 
$\delta = 2$, if $\varepsilon = -1$.  To construct specific examples, we will 
vary the parameters,~$r$,~$q$ and~$n$, keeping their principal significance. 
Thus~$q$ is a power of the prime~$r$, where $r \neq p$, and~$n$ is a positive 
integer.  Whenever we have chosen~$r$,~$q$ and~$n$, we adopt the corresponding 
notation introduced in Section~\ref{TheSpecialLinearAndUnitaryGroups}. 
In particular,~$\mathbb{F}$ denotes an
algebraic closure of the finite field with~$r$ elements. Moreover,
$\tilde{\bG} = \GL_n( \mathbb{F} )$, $\tilde{G} = \GL_n^{\varepsilon}( q ) 
= \tilde{\bG}^F$, $G = \SL_n^{\varepsilon}(q) = \tilde{\bG}^F$ and 
$Z = Z( G )$, where~$F$ is as in Notation~\ref{HypoI}(iii). These definitions 
are in accordance with Notation~\ref{HypoI}, except that we also allow $n = 1$ 
here. Furthermore, we put 
$\bG^* := \tilde{\bG}/Z( \tilde{\bG} ) = \PGL_n( \mathbb{F} )$,
with the Steinberg morphism induced from the one on~$\tilde{\bG}$. Then
$G^* = \tilde{G}/Z( \tilde{G} ) = \PGL_n( q )$. Note that the inclusion
$i: \bG \rightarrow \tilde{\bG}$ is a regular embedding and that the dual
epimorphism $i^* : \tilde{\bG}^* \rightarrow \bG^*$ is just the canonical
map, if we identify~$\tilde{\bG}$ with its dual group~$\tilde{\bG}^*$.
Our principal aim is to prove Theorem~\ref{MainTheoremSLNSUNII}.

We begin by constructing suitable prime powers~$q$.

\begin{lem}
\label{ExistenceOfQ}
Let~$a$ be a positive integer. Then there is a prime~$r$ and a power~$q$ of~$r$ 
such that $p^a \mid q - \varepsilon$ and $p^{a+1} \nmid q - \varepsilon$. If
$p^a = 3$, there is such a~$q$ with $q > 2$.
\end{lem}
\begin{proof}
Let~$r$ be a prime such that $p \mid r - \varepsilon$ but 
$p^2 \nmid r - \varepsilon$. Then put $q := r^{p^{a-1}}$. The last statement
is trivial.
\end{proof}

The following result considers the situation of 
Lemma~\ref{AnalysingTheDeltas}(c).

\begin{prop}
\label{Existence1Prop}
Let $a, a_1, c'$ be integers with $a, a_1$ positive 
and $0 \leq c' \leq a$. Let~$q$ satisfy the conclusion of 
{\rm Lemma~\ref{ExistenceOfQ}} with respect to~$p$ and~$a$, and put 
$n := p^{a+a_1}$. Then there is a 
block~$\bB$ of $G = \SL^\varepsilon_n( q )$ such that the following 
statements hold.

{\rm (a)} The defect group~$D$ of~$\bB$ is cyclic of order $p^{a+a_1}$.

{\rm (b)} There is a cyclic block of $\GL^\varepsilon_n( q )$ 
covering~$\bB$.

{\rm (c)} There is a subgroup $Y \leq Z$ with $|Y| = p^{c'}$, and a 
block~$\bar{\bB}$ of~$G/Y$ dominated by~$\bB$ and with defect group $D/Y$.
\end{prop}
\begin{proof}
Let~$\tilde{\bT}$ denote a Coxeter torus of~$\tilde{\bG}$; thus~$\tilde{T}$ is a
cyclic group of order~$q^n - \varepsilon$ (notice that~$n$ is odd and 
$n \geq 3$). Put $\bT := \tilde{\bT} \cap \bG$. By 
\cite[Lemma~$3.6.3$]{HL25}
there
exists $\tilde{s} \in \tilde{T}$, whose image in~$G^*$ is strictly regular.

Let~$\tilde{\chi}$ denote the corresponding irreducible Deligne-Lusztig 
character of~$\tilde{G}$ and let~$\tilde{\bB}$ denote the $p$-block 
of~$\tilde{G}$ containing~$\tilde{\chi}$. By 
\cite[Lemma~$2.5.16$]{HL25},
a Sylow $p$-subgroup~$\tilde{D}$ 
of~$\tilde{T}$ is a defect group of~$\tilde{\bB}$. In particular,~$\tilde{D}$ 
is cyclic of order $p^{2a+a_1}$.

As $s \in G^*$ is strictly regular, \cite[Corollary~$2.5.17$]{HL25}
implies that $\chi := \Res^{\tilde{G}}_G( \tilde{\chi} )$ is
irreducible,~$\tilde{\bB}$ covers a unique block~$\bB$ of~$G$ and that 
$D := \tilde{D} \cap G$ is a defect group of~$\bB$. Clearly,~$D$ is cyclic of 
order $p^{a+a_1}$. Let $Y \leq G$ with $|Y| = p^{c'}$. The $p$-block~$\bar{\bB}$ 
of~$G/Y$ dominated by~$\bB$ has defect group $D/Y$. This concludes our proof.
\end{proof}

Proposition~\ref{Existence1Prop} shows that all possible parameters determined 
in Lemma~\ref{AnalysingTheDeltas}(c) arise from blocks.
We next investigate the situation of Lemma~\ref{AnalysingTheDeltas}(d), where 
we restrict to the cases of Proposition~\ref{MainCorCase2}(a) with $n_1$, $n_2$ 
odd, and Proposition~\ref{MainCorCase2}(b) with $n_1$ even and $n_2$ odd.

\begin{prop}
\label{Existence2PropA}
Let $a, a_1, a_2$ be integers with $a, a_1$ positive and 
$0 \leq a_2 \leq a, a_1$. If $a_1 > a_2$ put $c := a_2$. If $a_1 = a_2$, choose 
an integer~$c$ with $a_1 \leq c \leq a$. If $a_1 = a_2 = c$, let $p > 3$.

Let~$q$ satisfy the conclusion of {\rm Lemma~\ref{ExistenceOfQ}} with respect 
to~$p$ and~$a$. If $p^a = 3$, assume that $q > 2$. Put $m_2 := 1$. If 
$a_1 > a_2$, let $m_1 := 1$. If $a_1 = a_2$, let $m_1 := p^{c-a_1} - 1$ unless 
$c = a_1$; in the latter case, let $m_1 := 2$ if $p \neq 3$, and $m_2 := 4$, if 
$p = 3$.  Finally, let $n := n_1 + n_2$ with $n_j := m_jp^{a_j}$ for $j = 1, 2$.

Let $Y := O_p( Z )$. Then $|Y| = p^c$ and there is a block~$\bar{\bB}$ of 
$\bar{G} = G/Y$ with cyclic defect group of order $p^{a + a_1 + a_2 - c}$. 
Moreover, the block $\bB$ of~$G$ dominating~$\bar{\bB}$ has abelian defect 
groups, which are direct products of two cyclic groups of order $p^{a+a_1}$ 
and $p^{a_2}$.
\end{prop}
\begin{proof}
Notice that $n = p^{a_1} + p^{a_2}$ if $a_1 > a_2$, and $n = p^c$ if 
$a_1 = a_2 < c$. If $a_1 = a_2 = c$, then $n = 3p^{a_2}$ if $p > 3$, and 
$n = 5p^{a_2}$ if $p = 3$. In any case, $|O_p( Z )| = p^c$.

Let~$V$ denote the natural vector space of~$\tilde{G}$. Write 
$V = V_1 \oplus V_2$, an orthogonal decomposition into non-degenerate subspaces 
if $\varepsilon = -1$, with $\dim V_j = n_j$ for $j = 1, 2$. Fix 
$j \in \{ 1, 2 \}$. Let $\tilde{G}_j$ denote the subgroup of~$\tilde{G}$ induced 
on~$V_j$, so that $\tilde{G}_j \cong \GL_{n_j}^\varepsilon( q )$. Choose a 
maximal torus $\tilde{T}_{j} \leq \tilde{G}_j$ such that $\tilde{T}_{j}$ is 
cyclic of order $q^{n_j} - \varepsilon^{n_j}$. Let~$\tilde{D}_{j}$ denote the 
Sylow $p$-subgroup of~$\tilde{T}_{j}$; then~$\tilde{D}_j$ is cyclic of 
order~$p^{a+a_j}$. 

Put $\tilde{D} = \tilde{D}_{1} \times \tilde{D}_{2}$ and
$\tilde{T} := \tilde{T_1} \times \tilde{T}_2 \leq 
\tilde{G}_1 \times \tilde{G}_2 \leq \tilde{G}$. Then $\tilde{T}$ is a maximal
torus of~$\tilde{G}$ corresponding to the partition of~$n$ with parts 
$n_1, n_2$.  If $n_2 = 1$, put $f_2 = 1$ and $\tilde{s}_2 = 1$. Otherwise, 
$n_j \geq 3$ for $j = 1, 2$, and the torus~$\tilde{T}_j$ contains a 
$p'$-element~$\tilde{s}_j$ of prime order $f_j$, such that 
$f_j \nmid q^i - \varepsilon^i$ for all $1 \leq i < n_j$; see 
\cite[Lemma~$3.6.3$]{HL25}.
By definition, $n_1 \neq n_2$, and hence $f_1 \neq f_2$.
Put $\tilde{s} := \tilde{s}_1\tilde{s}_2 \in \tilde{T}$. Then,~$\tilde{s}_1$ 
and~$\tilde{s}_2$ have no common eigenvalue, and thus
$C_{\tilde{G}}( \tilde{s} ) = 
C_{\tilde{G}_1}( \tilde{s}_1 ) \times C_{\tilde{G}_1}( \tilde{s}_1 ) = 
\tilde{T}_1 \times \tilde{T}_2 = \tilde{T}$. In other words,~$\tilde{s}$ is 
regular in~$\tilde{G}$. Let~$s$ denote the image of~$\tilde{s}$ in~$G^*$. Then
$|s| = f_1f_2$, and thus~$s$ is strictly regular in~$G^*$ 
by~\cite[$2.5.1$]{HL25}.
Let~$\tilde{\bB}$ denote the $p$-block of~$\tilde{G}$ containing the irreducible 
Deligne-Lusztig character $\tilde{\chi} \in \mathcal{E}( \tilde{G}, \tilde{s} )$. 
As in the proof of
Proposition~\ref{Existence1Prop} we find that
$\chi := \Res^{\tilde{G}}_G( \tilde{\chi} ) \in \Irr( G )$, and
the $p$-block~$\bB$ of~$G$ containing~$\chi$ is the unique block of~$G$ 
covered by~$\tilde{\bB}$. Also, $D := \tilde{D} \cap G$ is a defect group 
of~$\bB$.

For $j = 1, 2$, let $u_j$ denote a generator of $\tilde{D}_j$. By 
Lemma~\ref{DeterminantLem}(d), we have $|\det(u_j)| = p^{a}$ for $j = 1, 2$. 
If $a_1 = a_2$ we choose $u_1$ and $u_2$ in such a way that 
$(u_1u_2)^{p^{a_1}} \in Z$. This is possible since the eigenvalues of generators 
of~$\tilde{D}_1$ and~$\tilde{D}_2$ span the same subgroup of $\mathbb{F}^*$.
In any case, there is an integer~$e$, coprime to~$p$, such that $u_1u_2^{e}$ has 
determinant~$1$ and order $p^{a+a_1}$. By our choice of~$u_1$,~$u_2$, we may and 
will take $e = -m_1$ in case $a_1 = a_2$.
As $u_2^{p^a}$ has determinant~$1$ and order~$p^{a_2}$, we obtain
$D = \langle u_1u_2^{e} \rangle \times \langle u_2^{p^a} \rangle$. This gives our 
claim on the structure of~$D$.

We claim that $\bar{D} := D/Y$ is
cyclic. If $a_1 > a_2$, then $\langle u_1u_2^{e} \rangle \cap Y = \{ 1 \}$, 
since the non-trivial eigenvalues of $u_1^{p^{a_1}}$ and $u_2^{ep^{a_1}}$ have 
distinct orders. As $|D| = p^{a+a_1+a_2}$ and $a_2 = c$, we get 
$D = \langle u_1u_2^{e} \rangle \times Y$. In particular, $\bar{D}$ is cyclic.
If $a_1 = a_2$, we have $Y = \langle (u_1u_2)^{p^{a+a_1-c}} \rangle$ by our
choice of~$u_1$ and~$u_2$. The Frattini subgroup~$\Phi(D)$ of~$D$ is generated 
by $(u_1u_2^{e})^p$ and $u_2^{p^{a+1}}$ as a direct product. An elementary
calculation shows that $(u_1u_2)^{p^{a+a_1-c}}$ is not contained in $\Phi( D )$.
Indeed, $u_1u_2 = (u_1u_2^e)u_2^{1-e}$, and thus $(u_1u_2)^{p^{a+a_1-c}} \in 
\langle (u_1u_2^{e})^p \rangle \times \langle u_2^{p^{a+1}} \rangle$ would imply 
that $$u_2^{fp^{a+1}} = u_2^{(1-e)p^{a+a_1-c}} = u_2^{(1+m_1)p^{a+a_1-c}}$$
for some integer~$f$. However, the $p$-part of $1 + m_1$ equals~$p^{c - a_1}$.
Hence 
$$|u_2^{(1+m_1)p^{a+a_1-c}}| = |u_2^{p^a}| > |u_2^{fp^{a+1}}|,$$
a contradiction.
As~$D$ is a $2$-generator group, this implies that~$\bar{D}$ is cyclic.

Let~$\bar{\bB}$ denote 
the block of $\bar{G}$ dominated by~$\bB$. The defect group
of~$\bar{\bB}$ equals~$\bar{D}$, which proves our assertions.
\end{proof}

We next investigate the situation of Lemma~\ref{AnalysingTheDeltas}(d)(iii).
\begin{prop}
\label{Existence2PropB}
Let~$p$ be an odd prime and let $c', c, a$ be integers with 
$0 \leq c' \leq c < a$. If $c = 0$, put $n := 2$. Otherwise, let
$n \in \{ p^c, 2p^c \}$. 

Let~$q$ satisfy the conclusion of {\rm Lemma~\ref{ExistenceOfQ}} with respect
to~$p$ and~$a$.  Then there is a block~$\bB$ of $G = \SL^\varepsilon_n( q )$ 
such that the following statements hold.

{\rm (a)} The defect group of~$\bB$ is cyclic of order $p^{a}$.

{\rm (b)} There is a subgroup $Y \leq Z$ with $|Y| = p^{c'}$, and a
block~$\bar{\bB}$ of $G/Y$ dominated by~${\bB}$.
\end{prop}
\begin{proof}
The statements are clear if $c = 0$ and $n = 2$. Thus assume that $c > 0$ in the
following. Put $n_1 := n - 1$ and $n_2 := 1$. Let 
$G_1 := \GL_{n_1}^\varepsilon( q )$, naturally embedded into~$G$. Let 
$D := O_p( Z( G_1 ) )$. Then~$D$ is cyclic of order~$p^a$. As $c < a$, we have 
$C_G(D) = G_1$. If $n_1 > 2$, a cyclic maximal torus of~$G_1$ of order 
$q^{n_1} - \varepsilon^{n_1}$ contains $p'$-elements which are regular with 
respect to~$G_1$; see \cite[Lemma~$3.6.3$]{HL25}.
In this case, \cite[Lemma~$2.5.16$]{HL25}
guarantees the existence of a block~$\bb$ 
of~$G_1$ with defect group~$D$. The same conclusion clearly also holds for 
$n_1 = 2$. Now $N_G(G_1) = G_1$, and thus the Brauer correspondent~$\bB$ 
of~$\bb$ satisfies~(a).

The proof of~(b) is analogous to the proof of~(c) of 
Proposition~\ref{Existence1Prop}.
\end{proof}

We finally investigate the situation of Lemma~\ref{AnalysingTheDeltas}(e),
where we restrict to the cases of Proposition~\ref{MainCorCase3}(b)(c), also
assuming that $a \neq a_1$. (If $a = a_1$ in Proposition~\ref{MainCorCase3}(b),
the resulting set~$A$ is an interval.)

\begin{prop}
\label{Existence3Prop}
Let~$p$ be an odd prime and let $a, a_1$ be positive integers with $a \neq a_1$.

Let~$q$ satisfy the conclusion of {\rm Lemma~\ref{ExistenceOfQ}} with respect
to~$p$ and~$a$. Put $n_1 := p^{a_1}$ and $n_3 := 1$. If $a_1 > a$, put 
$n_2 := 2p^a - 1$, and if $a_1 < a$, put $n_2 := p^a - p^{a_1} - 1$. In any 
case, let $n := n_1 + n_2 + n_3$, so that $p^a = |O_p( Z )|$. 

Put $Y := O_p( Z )$.  Then there is a block~$\bar{\bB}$ of 
$\bar{G} := G/Y$ with cyclic defect group of order $p^{a+a_1}$. Moreover, the 
defect groups of the block $\bB$ of~$G$ dominating~$\bar{\bB}$ are direct 
products of two cyclic groups of orders $p^{a+a_1}$ and $p^{a}$, respectively.
\end{prop}
\begin{proof}
Let~$V$ denote the natural vector space of~$\tilde{G}$. Write
$V = V_1 \oplus V_2 \oplus V_3$, an orthogonal decomposition into non-degenerate
subspaces if $\varepsilon = -1$, with $\dim V_j = n_j$ for $j = 1, 2, 3$. Fix 
$j \in \{ 1, 2, 3 \}$. Let $\tilde{G}_j$ denote the subgroup of~$\tilde{G}$
induced on~$V_j$, so that $\tilde{G}_j \cong \GL_{n_j}^\varepsilon( q )$. Choose 
a maximal torus $\tilde{T}_{j} \leq \tilde{G}_j$ such that $\tilde{T}_{j}$ is 
cyclic of order $q^{n_j} - \varepsilon^{n_j}$. Let~$\tilde{D}_{j}$ denote the 
Sylow $p$-subgroup of~$\tilde{T}_{j}$; then~$\tilde{D}_j$ is cyclic of 
order~$p^{a+a_j}$ (with $a_2 = a_3 = 0$).

Put $\tilde{D} = \tilde{D}_{1} \times \tilde{D}_{2} \times \tilde{D}_3$ and
$\tilde{T} := \tilde{T_1} \times \tilde{T}_2  \times \tilde{T}_3 \leq 
\tilde{G}_1 \times \tilde{G}_2  \times \tilde{G}_3 \leq \tilde{G}$. Then 
$\tilde{T}$ is a maximal torus of~$\tilde{G}$ corresponding to the partition 
of~$n$ with parts $n_1, n_2, n_3$.

For $j = 1, 2$, the
torus~$\tilde{T}_j$ contains a $p'$-element~$\tilde{s}_j$ of prime order $f_j$,
such that $f_j \nmid q^i - \varepsilon^i$ for all $1 \leq i < n_j$; see
\cite[Lemma~$3.6.3$]{HL25}.
Now $f_1 \neq f_2$, as $n_1 \neq n_2$.
In particular,~$\tilde{s}_1$ and~$\tilde{s}_2$ have no common eigenvalue, and no 
eigenvalue~$1$. Put $\tilde{s} := \tilde{s}_1\tilde{s}_2 \in \tilde{T}$. 
Then $C_{\tilde{G}}( \tilde{s} ) = 
C_{\tilde{G}_1}( \tilde{s}_1 ) \times C_{\tilde{G}_2}( \tilde{s}_2 ) \times
\tilde{G}_3 = \tilde{T}_1 \times \tilde{T}_2 \times \tilde{T}_3 = \tilde{T}$. In 
other words,~$\tilde{s}$ is regular in~$\tilde{G}$. Let~$s$ denote the image 
of~$\tilde{s}$ in~$G^*$. Then $|s| = f_1f_2$, and thus~$s$ is strictly regular 
in~$G^*$ by~\cite[$2.5.1$]{HL25}.

Let~$\tilde{\bB}$ denote the $p$-block of~$\tilde{G}$ containing the 
corresponding irreducible Deligne-Lusztig character 
$\tilde{\chi} \in \mathcal{E}(G,s)$. As in the proof of 
Proposition~\ref{Existence1Prop} we find that
$\chi := \Res^{\tilde{G}}_G( \tilde{\chi} ) \in \Irr( G )$, and
the $p$-block~$\bB$ of~$G$ containing~$\chi$ is the unique block of~$G$
covered by~$\tilde{\bB}$. Also, $D := \tilde{D} \cap G$ is a defect group
of~$\bB$.

Clearly, $\tilde{D} \cap G \cong \tilde{D}_1 \times \tilde{D}_2$, and thus
the structure of~$D$ is as claimed. there is an element $t \in D$ of 
order~$p^{a+a_1}$, which acts trivially on~$V_2$, and whose eigenvalue on~$V_3$ 
has order~$p^a$. In particular, $\langle t \rangle \cap Y = \{ 1 \}$, so that 
$D = \langle t \rangle \times Y$. Hence $D/Y$ is cyclic of order $p^{a+a_1}$.

The block $\bar{\bB}$ of $\bar{G}$ dominated by~$\bB$ has defect group $D/Y$, 
which proves our assertions.
\end{proof}

Taking 
$\varepsilon = -1$, Propositions~\ref{Existence1Prop}--\ref{Existence3Prop},
together with Propositions~\ref{Case1Prop}--\ref{MainCorCase3}, prove 
Theorem~\ref{MainTheoremSLNSUNII}.

\section{Glossary}\label{sec:Glossary}%
To finish with, to facilitate the reading of our results, we summarize the main 
notation used in this manuscript in form of a glossary. Part of this notation 
was introduced in Part I and Part II 
of our work. 

\medskip

\noindent\textbf{General assumptions}:
\begin{enumerate}[label=$\bullet$, leftmargin=5mm]
	\item  $p$ is an odd prime number;
	\item  $(K,\mathcal{O},k)$ is a sufficiently large $p$-modular system, 
               where $\cO$ is a d.v.r. of characteristic zero with residue 
               field $k=\overline{k}$ of characteristic $p$;
	\item  $l$ is a non-negative integer.
\end{enumerate}
\smallskip

\noindent\textbf{Special functions and intervals}:
\begin{enumerate}[label=$\bullet$, leftmargin=5mm] 
	\item  $\sigma^{}_{\chi}:=\mathrm{sgn}\circ \chi: X\rightarrow \{-1,0,1\}$ 
               is the sign function associated to the map 
               $\chi:X\rightarrow \mathbb{R}$, where $X$ is a set 
               (see \cite[Definition 2.1.1]{HL25});
	\item  $\rho^{[m]}( t ) := (\rho( t^{p^{m-1}} ), \rho( t^{p^{m-2}} ), 
               \ldots, \rho( t^p ), \rho( t ) )\in X^{m}$ for a set $X$ and 
               positive integer $m$, lists  the values of a map 
               $\rho:H\rightarrow X$ at the $p$-elements 
               $t, t^p, t^{p^2}, \ldots, t^{p^{m-1}}$ of a finite group $H$  in 
               reverse order (see \cite[Definition 2.1.2]{HL25});
	\item  $\Lambda:=\{x\in \mathbb{Z}\mid 1\leq x\leq l-1\}$;
	\item  an \emph{interval} is a subset of $\Lambda$ which is the intersection 
               of $\Lambda$ with an interval of $\mathbb{R}$, possibly the empty set 
               (see \cite[Subsection 2.2]{HL25});
	\item  a non-empty interval is written as $[i,j]$ with~$i$, respectively~$j$, its 
               smallest, respectively largest, element; 
               (see \cite[Subsection 2.2]{HL25});
	\item  $\mathbb{F}_{2}=\{0,1\}$ is the field with $2$ elements;
	\item  $\mathbf{1}_A$ is the characteristic function of $A \subseteq \Lambda$, 
               i.e. the element 
               $\mathbf{1}_A = (\alpha_0, \alpha_1, \ldots, \alpha_{l-1} )\in 
               \mathbb{F}_2^\Lambda$ with $\alpha_j = 1$ if and only if  $j \in A$ 
               (see \cite[Subsection 2.2]{HL25});
	\item  $\omega_{\Lambda}: \mathbb{F}_2^{\Lambda} \rightarrow \{ -1, 1 \}^l$ 
               is the $\mathbb{F}_{2}$-isomorphism defined by 
               $$
               \phantom{\bullet\quad}\omega_{\Lambda}(\alpha_0, \ldots , \alpha_{l-1})_i = 
               \begin{cases} 
               +1, \text{\ if\ } \sum_{j = 0}^{i-1} \alpha_j = 0\\
               -1, \text{\ if\ } \sum_{j = 0}^{i-1} \alpha_j = 1
               \end{cases}\quad\quad\text{for each\ } 1 \leq i \leq l;
               $$
                (see \cite[Subsection 2.2]{HL25}).
\end{enumerate}
\smallskip

\noindent\textbf{Endo-permutation module associated to a cyclic block 
                 $\mathcal{B}$ with defect group $D$ of order $p^l$}:
\begin{enumerate}[label=$\bullet$, leftmargin=5mm]
	\item  $W(\mathcal{B})$ is the endo-permutation $kD$-module 
               associated to $\mathcal{B}$, uniquely determined, as an element 
               of the Dade group of $D$, by the $l$-tuple
               $(\alpha_0, \alpha_1, \ldots, \alpha_{l-1} )\in \mathbb{F}_2^{\Lambda}$,  
               also written 
               $W(\mathcal{B})=W_{D}(\alpha_0, \alpha_1, \ldots, \alpha_{l-1})$  
               (see \cite[Subsection 3.1]{HL24}); 
	\item  the label $(\alpha_0, \alpha_1, \ldots, \alpha_{l-1})$ above is 
               identified with a subset of $\Lambda$ via the $\mathbb{F}_{2}$-isomorphism  
               $\mathcal{P}(\Lambda)\rightarrow  \mathbb{F}_2^{\Lambda}, 
               A\mapsto \mathbf{1}_{A}$  (see \cite[Subsections 2.2 and 2.3]{HL25});
	\item  $W_D(A) := W_D(\mathbf{1}_A)$  for any subset $A\subseteq \Lambda$ 
               (see \cite[Definition~2.3.1]{HL25});
        \item  $W_D( \emptyset ) \cong k$, the trivial module  (particular case of the above);
	\item  $\omega_{W}:=\sigma_{\rho^{}_{W}}$ is the sign function associated 
               to  the $K$-character $\rho^{}_{W}$ afforded by the unique lift of 
               determinant one of $W:=W(\mathcal{B})$ to $\cO$  
               (see \cite[Subsection 3.1]{HL24}); 
	         
	\item  $\omega_{W}^{[l]}(t)$ with $t$ a generator of $D$ determines $W$ 
               up to isomorphism (see \cite[Lemma~2.3.2]{HL25});
	\item  $\omega_W^{[l]}( t )=\sigma_{\chi}^{[l]}( t )$  provided $t$ is 
               as above, $\langle t^{p^{l-1}} \rangle$  is in the center of the 
               group considered, and  $\chi\in \Irr(\mathcal{B})$ denotes the 
               unique non-exceptional character  (see \cite[Remark~2.3.3]{HL25});
	\item  $\omega_{\mathbf{H}}( s ) := 
               \varepsilon_{\mathbf{H}}\varepsilon_{C^{\circ}_{\mathbf{H}}(s)}$, 
               when it occurs, is a sign associated with the algebraic group 
               $\mathbf{H}$ under consideration  and any  semisimple element 
               $s\in H$ (see \cite[Definition 2.5.6]{HL25}). 
\end{enumerate}
\smallskip

\noindent\textbf{The groups considered (Section~\ref{sec:Intro}, 
                 Section~\ref{sec:Analysis} and Section~\ref{sec:Synthesis})}:
\begin{enumerate}[label=$\bullet$, leftmargin=5mm]
        \item  $\varepsilon\in\{\pm 1\}$; $\delta := 1$ if $\varepsilon = 1$, and  
               $\delta := 2$ if $\varepsilon = -1$;
	\item  $q$ is a power of a prime number $r$ such that $p \mid q-\varepsilon$;
	\item  $\mathbb{F}$ is an algebraic closure of the finite field with~$r$ elements;
	\item  $n$ is a positive integer satisfying: $n\geq 2$ in Section~\ref{sec:Intro} 
               and Section~\ref{sec:Analysis}, and $n\geq 1$ in Section~\ref{sec:Synthesis};
	\item  $\tilde{\bG} := \GL_n( \mathbb{F} )$ and 
               $\bG :=\{ g \in \tilde{\bG} \mid \det(g) = 1 \}$; 
	\item  $F := F_{\varepsilon}$ is a Steinberg morphism of~$\tilde{\bG}$ such that 
               {$\tilde{G} := \tilde{\bG}^F = \GL_n^\varepsilon( q )$};
	\item  $G := \bG^F = \SL_n^\varepsilon( q )$;
	\item  $Z := Z( G )$ and satisfies $|Z|=\mathrm{gcd}(q-\varepsilon,n)$;
        \item  $a, b, c$ are non-negative integers such that $p^a, p^b, p^c$
               are the highest powers of~$p$ dividing $q - \varepsilon$, $n$ and
               $\mathrm{gcd}(q-\varepsilon,n)$, respectively. Thus $a > 0$ and
               $c = \min \{ a, b \}$;
	\item  $\bV := \mathbb{F}^n$ is the natural vector space for~$\tilde{\bG}$, 
               and $V := \mathbb{F}_{q^\delta}^n \subseteq \bV$;
	\item  $\bG^* := \tilde{\bG}/Z( \tilde{\bG} ) = \PGL_n( \mathbb{F} )$ 
               (with the Steinberg morphism induced from the one on~$\tilde{\bG}$);
	\item  $G^* := \tilde{G}/Z( \tilde{G} ) = \PGL_n( q )$;
	\item  $Y\leq Z$ is a $p$-subgroup and  $\bar{G}:=G/Y$ is a central quotient 
               of~$G$.
\end{enumerate}
\smallskip

\noindent\textbf{The blocks and the associated local configuration (Section~\ref{sec:Analysis})}: 
\begin{enumerate}[label=$\bullet$, leftmargin=5mm]
        \item  $\bar{\bB}$ is a cyclic $p$-block of~$\bar{G}$ and $\bB$ is the unique 
               block of $G$ dominating $\bar{\bB}$;
        \item  $D$ is a defect group of $\bB$ such that $\bar{D} := D/Y$ is 
               a defect group of $\bar{\bB}$;
        \item  $t \in D$ is such that $\bar{t} := tY \in \bar{G}$ 
               generates~$\bar{D}$;
        \item  $h \in \{ 1, 2, 3 \}$ is the number of irreducible factors of the minimal 
               polynomial of~$t$ acting on~$V$; 
        \item  $c'$ is the non-negative integer such that $|\bar{D}| = |t|/p^{c'}$.
        \item  $C := C_G( D )$ and $\tilde{C} := C_{\tilde{G}}( D )$;
        \item  $\bc$ is a Brauer correspondent of~$\bB$ in~$C$ and $\tilde{\bc}$
               is a block of $\tilde{C}$ covering~$\bc$; see Figure~\ref{DiagramForSLn};
        \item  $\tilde{D}$ is a defect group of~$\tilde{\bc}$ with $D = C \cap \tilde{D}
               = G \cap \tilde{D}$.
\end{enumerate}

\section*{Acknowledgements}

The authors thank Olivier Dudas, Meinolf Geck, Radha Kessar, 
Burkhard K{\"u}ls\-ham\-mer,
Markus Linckelmann, Frank L{\"u}beck,
Klaus Lux, Gunter Malle and Jay Taylor for innumerable invaluable discussions
and elaborate explanations on various aspects of this work.


\begin{thebibliography}{BMM93}
\bibitem[AlBr79]{AlpBro} J.~Alperin and M.~Brou{\'e}, \emph{Local methods in 
        block theory}, Ann.\ Math.\ \textbf{110} (1979), 143--157.

\bibitem[CaEn99]{CaEn99} M.~Cabanes and M.~Enguehard, \emph{On blocks of finite 
        reductive groups and twisted induction}, Adv.\ Math.\ \textbf{145} 
        (1999), 189--229. 

\bibitem[FoSr82]{fs2} P.~Fong and B.~Srinivasan, \emph{The blocks of finite 
        general linear and unitary groups}, Invent.\ Math.\ \textbf{69} 
        (1982), 109--153.

\bibitem[GeMa20]{GeMa} M.~Geck and G.~Malle, \emph{The Character Theory of 
        Finite Groups of Lie Type, A Guided Tour}, Cambridge Studies in Advanced
        Mathematics, 187, Cambridge University Press, Cambridge, 2020.

\bibitem[HK85]{HK85} M.~E.~Harris and R.~Kn{\"o}rr, \emph{Brauer correspondence 
        for covering blocks of finite groups}, Comm.\ Algebra \textbf{13} (1985), 
	1213--1218.

\bibitem[HL24]{HL24} G.~Hiss and C.~Lassueur, \emph{On the source algebra 
        equivalence class of blocks with cyclic defect groups, I}, 
        Beitr.\ Algebra\ Geom.\ \textbf{65} (2024), 187--207.

\bibitem[HL25]{HL25} G.~Hiss and C.~Lassueur, \emph{On the source algebra 
        equivalence class of blocks with cyclic defect groups, II}, preprint.

\bibitem[Lin18]{LinckBook} M.~Linckelmann, \emph{The Block Theory of Finite 
        Group Algebras, {V}ol.~{II}}, London Mathematical Society Student Texts, 
        vol.~92, Cambridge University Press, Cambridge, 2018.   

\bibitem[Lue21a]{LL2} F.~L{\"u}beck, \emph{Character Degrees and their 
        Multiplicities for some Groups of Lie Type of Rank $< 9$},
        \url{http://www.math.rwth-aachen.de/~Frank.Luebeck/chev/DegMult/index.html}.

\bibitem[NT89]{NaTs} H.~Nagao and Y.~Tsushima, \emph{Representations of Finite
        Groups}, Academic Press, Inc., Boston, MA, 1989.
\end{thebibliography}
\end{document}